\documentclass[11pt,letterpaper]{amsart}%
\usepackage{amsfonts}
\usepackage{amsmath}
\usepackage{amssymb}
\usepackage{color}
\usepackage{graphicx}%
\providecommand{\U}[1]{\protect\rule{.1in}{.1in}}
\newtheorem{theorem}{Theorem}[section]
\theoremstyle{plain}

\newtheorem{lemma}{Lemma}[section]

\newtheorem{remark}{Remark}

\numberwithin{equation}{section}
\begin{document}
\title[Optimal dimension-dependent estimate of the second-order discrete Riesz transforms]{Optimal dimension-dependent $\ell^p$ and $\ell^{1,\infty}$ estimates of the second-order discrete Riesz transforms}

\author{Hanli Tang}
\address[Hanli Tang]{Laboratory of Mathematics and Complex Systems (Ministry of Education), School of Mathematical Sciences, Beijing Normal University, Beijing, 100875, China}
\email{hltang@bnu.edu.cn}

\author{Zewei Xu}
\address[Zewei Xu]{Laboratory of Mathematics and Complex Systems (Ministry of Education), School of Mathematical Sciences, Beijing Normal University, Beijing, 100875, China}
\email{zwxu@mail.bnu.edu.cn}

\keywords{optimal, dimension-dependent, second-order discrete Riesz transform}
\thanks{The first author
was partly supported by  National Natural Science Foundation of China (Grant No.12471100).}


\begin{abstract}
In this paper we investigate the optimal
dimension-dependent estimates of the second-order discrete
Riesz transforms
\[
R_{\mathrm{dis}}^{(jk)}f(n)
=
c_d\sum_{m\in\mathbb Z^d\setminus\{0\}}
\frac{m_jm_k}{|m|^{d+2}}f(n-m),
\qquad
c_d=\frac{\Gamma\left(\frac{d+2}{2}\right)}{\pi^{d/2}}.
\]

For $j\neq k$ and every fixed $1<p<\infty$, we prove that
\[
\|R_{\mathrm{dis}}^{(jk)}\|_{\ell^p\to\ell^p}
=
c_d\left[
\frac{2}{2^{d/2}}
+
\left(\frac83+o(1)\right)\frac{d}{3^{d/2}}
\right]
\]
and
\[
\frac{2c_d}{2^{d/2}}
\leq
\|R_{\mathrm{dis}}^{(jk)}\|_{\ell^1\to\ell^{1,\infty}}
\leq
c_d\left[
\frac{2}{2^{d/2}}
+
\left(\frac83+o(1)\right)\frac{d}{3^{d/2}}
\right].
\]
Since
$
c_d\sim
\sqrt{\pi d}\left(\frac{d}{2\pi e}\right)^{d/2}
$ by Stirling's formula,
thl $\ell^p$ estimates give a negative answer to the conjecture proposed by
Ba\~nuelos and Kim  in \cite{BK2}.

The diagonal case exhibits quite a different phenomenon:
for every $j$ and every $1<p<\infty$, $R_{\mathrm{dis}}^{(jj)}$ is
neither bounded from $\ell^p(\mathbb Z^d)$ to  $\ell^p(\mathbb Z^d)$ nor of weak type $(1,1)$.
Cancellation is restored for the operators
$R_{\mathrm{dis}}^{(jj-kk)}
=R_{\mathrm{dis}}^{(jj)}-R_{\mathrm{dis}}^{(kk)}$. For every fixed
$1<p<\infty$,
\[
\|R_{\mathrm{dis}}^{(jj-kk)}\|_{\ell^p\to\ell^p}
=
4c_d\left[
1+(1+o(1))\frac{d}{2^{d/2}}
\right],
\]
and
\[
4c_d
\leq
\|R_{\mathrm{dis}}^{(jj-kk)}\|_{\ell^1\to\ell^{1,\infty}}
\leq
4c_d\left[
1+(1+o(1))\frac{d}{2^{d/2}}
\right].
\]
\end{abstract}
\maketitle
\section{Introduction}
The Riesz transforms are fundamental examples of Calder\'on--Zygmund
singular integral operators. The systematic theory of such operators,
including their weak-type endpoint estimates, originates in the
foundational work of Calder\'on and Zygmund \cite{CZ}. For
$k=1,\ldots,d$, the classical first-order Riesz transform on
$\mathbb R^d$ is defined by
\[
R^{(k)}f(x)
=
\operatorname{p.v.}\,
\widetilde c_d
\int_{\mathbb R^d}
\frac{y_k}{|y|^{d+1}}f(x-y)\,dy,
\qquad
\widetilde c_d
=
\frac{\Gamma\left(\frac{d+1}{2}\right)}
{\pi^{(d+1)/2}}.
\]
For $1<p<\infty$, set
\[
p^*=\max\left\{p,\frac{p}{p-1}\right\}.
\]
The sharp strong-type norm is dimension free:
\[
\|R^{(k)}\|_{L^p(\mathbb R^d)\to L^p(\mathbb R^d)}
=
\cot\left(\frac{\pi}{2p^*}\right),
\]
as proved by Iwaniec and Martin \cite{IM}; a probabilistic proof based on
sharp martingale inequalities was later given by Ba\~nuelos and Wang
\cite{BW}.

At the endpoint, the Calder\'on--Zygmund theory gives
$R^{(k)}:L^1(\mathbb R^d)\to L^{1,\infty}(\mathbb R^d)$.  The dependence
of the weak-type constant on the dimension was a long-standing problem.
Janakiraman \cite{Ja} obtained the bound $C\log d$ for each component, and
Spector and Stockdale \cite{SS} subsequently gave a reduction principle
and a new proof of this best-known dimensional estimate. Very recently,
Ouyang, Spector and Stockdale \cite{OSS} resolved Stein's dimension-free
problem by proving the stronger vector-valued estimate
\[
\left\|
\left(\sum_{k=1}^d|R^{(k)}f|^2\right)^{1/2}
\right\|_{L^{1,\infty}(\mathbb R^d)}
\leq
2\|f\|_{L^1(\mathbb R^d)}.
\]
In particular, every individual continuous Riesz transform has weak-type
$(1,1)$ norm at most $2$, independently of $d$. In dimension one,
$R^{(1)}$ is the Hilbert transform $H$.  Its weak-type norm is known
exactly: Davis \cite{Da} proved
\[
\|H\|_{L^1(\mathbb R)\to L^{1,\infty}(\mathbb R)}
=
\frac{\displaystyle\sum_{m=0}^{\infty}(2m+1)^{-2}}
{\displaystyle\sum_{m=0}^{\infty}(-1)^m(2m+1)^{-2}}
=
\frac{\pi^2}{8G}
\approx1.347,
\]
where $G$ is Catalan's constant.

We now pass to the continuous second-order theory, which is the principal
continuous model for the present paper. The classical second-order Riesz
transforms on $\mathbb R^d$ are given by
\begin{equation}\label{eq:continuous-second-order}
R^{(jk)}f(x)
=
c_d\,\operatorname{p.v.}
\int_{\mathbb R^d}
\frac{y_jy_k}{|y|^{d+2}}f(x-y)\,dy,
\qquad
c_d
=
\frac{\Gamma\left(\frac{d+2}{2}\right)}{\pi^{d/2}},
\end{equation}
and equivalently by
\[
\widehat{R^{(jk)}f}(\xi)
=
-\frac{\xi_j\xi_k}{|\xi|^2}\widehat f(\xi).
\]
If
\[
P_t(x)
=
(2\pi t)^{-d/2}e^{-|x|^2/(2t)}
\]
is the heat kernel and $T_tf=P_t*f$, then
\[
R^{(jk)}f(x)
=
\int_0^\infty
\frac{\partial^2T_tf(x)}
{\partial x_j\partial x_k}\,dt
=
\frac{\partial^2}{\partial x_j\partial x_k}
(-\Delta)^{-1}f(x).
\]

There are three natural second-order families, and all three have
dimension-free strong-type behavior in the continuous setting. For
$j\neq k$,
\begin{equation}\label{eq:continuous-offdiag-norm}
\|2R^{(jk)}\|_{L^p(\mathbb R^d)\to L^p(\mathbb R^d)}
=
p^*-1,
\end{equation}
while
\begin{equation}\label{eq:continuous-difference-norm}
\|R^{(jj)}-R^{(kk)}\|_{L^p(\mathbb R^d)\to L^p(\mathbb R^d)}
=
p^*-1.
\end{equation}
The sharpness of these estimates was proved in \cite{GMSS}. For the
diagonal components,
\begin{equation}\label{eq:continuous-diagonal-norm}
\|R^{(jj)}\|_{L^p(\mathbb R^d)\to L^p(\mathbb R^d)}
=
\gamma(p),
\end{equation}
where $\gamma(p)$ is the sharp Choi constant associated with
nonsymmetric martingale transforms; see \cite{BO}.  Thus, from the point
of view of $L^p$ norms, the off-diagonal transforms, the diagonal
transforms, and the traceless differences all remain uniformly controlled
with respect to dimension.

We next recall briefly what is known for the first-order direct
discretization. Define
\[
R_{\mathrm{dis}}^{(k)}f(n)
=
\widetilde c_d
\sum_{m\in\mathbb Z^d\setminus\{0\}}
\frac{m_k}{|m|^{d+1}}f(n-m).
\]
When $d=1$, this is precisely the discrete Hilbert transform
\[
H_{\mathrm{dis}}f(n)
=
\frac1\pi
\sum_{m\in\mathbb Z\setminus\{0\}}
\frac{f(n-m)}{m}.
\]
The long-standing problem of its exact $\ell^p$ norm was settled by
Ba\~nuelos and Kwa\'snicki \cite{BK}, who proved
\[
\|H_{\mathrm{dis}}\|_{\ell^p(\mathbb Z)\to\ell^p(\mathbb Z)}
=
\|H\|_{L^p(\mathbb R)\to L^p(\mathbb R)}
=
\cot\left(\frac{\pi}{2p^*}\right).
\]
Their argument uses a probabilistic discrete Hilbert transform and a
probability-kernel factorization, providing a striking instance in which
direct discretization preserves the sharp continuous norm.

Ba\~nuelos, Kim and Kwa\'snicki \cite{BKK} extended this probabilistic
construction to first-order Riesz transforms in higher dimensions.  Their
probabilistic transforms $T_R^{(k)}$ satisfy
\[
\|T_R^{(k)}\|_{\ell^p(\mathbb Z^d)\to\ell^p(\mathbb Z^d)}
=
\cot\left(\frac{\pi}{2p^*}\right),
\]
and, motivated by the discrete Hilbert-transform identity, they
conjectured that the directly sampled transforms
$R_{\mathrm{dis}}^{(k)}$ should have the same dimension-free norm.  Our
recent work \cite{STX} shows that this expectation fails for direct
sampling: for every fixed $1<p<\infty$,
\[
\|R_{\mathrm{dis}}^{(k)}\|_{\ell^p\to\ell^p}
=
2\widetilde c_d
\left(
1+(\sqrt2+o(1))\frac{d}{2^{d/2}}
\right),
\qquad d\to\infty.
\]
Thus already at first order the direct lattice kernel can behave very
differently from its continuous and probabilistic counterparts.  This
observation provides part of the motivation for examining the second-order
problem more closely.

The direct second-order discrete Riesz transform obtained by sampling the
kernel in \eqref{eq:continuous-second-order} is
\begin{equation}\label{eq:direct-second-order}
R_{\mathrm{dis}}^{(jk)}f(n)
=
\sum_{m\in\mathbb Z^d}
K_{\mathrm{dis}}^{(jk)}(m)f(n-m),
\end{equation}
where
\begin{equation}\label{eq:direct-second-kernel}
K_{\mathrm{dis}}^{(jk)}(m)
=
c_d\frac{m_jm_k}{|m|^{d+2}}
\mathbf 1_{\mathbb Z^d\setminus\{0\}}(m).
\end{equation}

Recently, Ba\~nuelos and Kim \cite{BK2} developed a probabilistic theory
of second-order discrete Riesz transforms. Their construction is the
second-order counterpart of the probabilistic models used for the
Hilbert and first-order Riesz transforms, with the periodic Poisson kernel
replaced by the periodic heat kernel
\begin{equation}\label{eq:periodic-heat-kernel}
H_t(x)
=
\sum_{n\in\mathbb Z^d}P_t(x-n),
\qquad
x\in\mathbb R^d,\quad t>0.
\end{equation}
Using a spacetime Doob $h$-process with $h=H_t$ and conditional
expectations of martingale transforms, they constructed the probabilistic
discrete second-order Riesz transforms $\mathcal R^{(jk)}$:
\begin{equation}\label{eq:probabilistic-second-discrete}
\mathcal R^{(jk)}f(n)
=
\sum_{m\in\mathbb Z^d}
\mathcal K^{(jk)}(n-m)f(m),
\end{equation}
where
\begin{equation}\label{eq:probabilistic-second-kernel}
\mathcal K^{(jk)}(m)
=
-
\int_0^\infty\int_{\mathbb R^d}
H_s(x)
\frac{\partial}{\partial x_j}
\left(
\frac{P_s(x)}{H_s(x)}
\right)
\frac{\partial}{\partial x_k}
\left(
\frac{P_s(x-m)}{H_s(x)}
\right)
\,dx\,ds.
\end{equation}

The martingale construction reproduces the sharp continuous constants.
For $j\neq k$,
\begin{equation}\label{eq:probabilistic-offdiag}
\|2\mathcal R^{(jk)}\|_{\ell^p\to\ell^p}
=
p^*-1
=
\|2R^{(jk)}\|_{L^p\to L^p},
\qquad j\neq k,
\end{equation}
and
\begin{equation}\label{eq:probabilistic-difference}
\|\mathcal R^{(jj)}-\mathcal R^{(kk)}\|_{\ell^p\to\ell^p}
=
p^*-1
=
\|R^{(jj)}-R^{(kk)}\|_{L^p\to L^p}.
\end{equation}
For the diagonal transforms they proved
\begin{equation}\label{eq:probabilistic-diagonal}
\|\mathcal R^{(jj)}\|_{\ell^p\to\ell^p}
\leq
\gamma(p)
=
\|R^{(jj)}\|_{L^p\to L^p}.
\end{equation}

There is further evidence that the probabilistic and directly sampled
models should be closely related.  Ba\~nuelos and Kim proved that, for
$m\neq0$,
\begin{equation}\label{eq:U-comparison}
\mathcal K^{(jk)}(m)
=
U(m)K_{\mathrm{dis}}^{(jk)}(m),
\end{equation}
where $U(m)$ is explicit and satisfies $U(m)\to-1$ rapidly as
$|m|\to\infty$. In the off-diagonal case, if
\[
\mathcal J^{(jk)}(m)
=
\bigl(U(m)+1\bigr)K_{\mathrm{dis}}^{(jk)}(m),
\]
then
\[
\bigl(\mathcal R^{(jk)}+R_{\mathrm{dis}}^{(jk)}\bigr)f
=
\mathcal J^{(jk)}*f,
\qquad
\mathcal J^{(jk)}\in\ell^1(\mathbb Z^d).
\]
Thus the probabilistic transform and the direct discretization have the
same leading kernel at large lattice distances and differ, up to sign, by
an $\ell^1$ perturbation.  Combined with
\eqref{eq:probabilistic-offdiag}--\eqref{eq:probabilistic-diagonal}, this
makes the possibility of sharp norm identification quite natural rather
than merely formal: the probabilistic model already has the continuous
sharp constants, while the discrepancy between the two discrete kernels
is summable.

A parallel relation holds on the continuous side. Ba\~nuelos and Kim
extend the expression in \eqref{eq:probabilistic-second-kernel} from
lattice points to $x\in\mathbb R^d$ and define
\begin{equation}\label{eq:probabilistic-continuous-kernel}
\mathbf K^{(jk)}(x)
=
\mathcal K^{(jk)}(x)\mathbf 1_{\{|x|\geq1\}}
-
c_d\frac{x_jx_k}{|x|^{d+2}}
\mathbf 1_{\{|x|<1\}},
\end{equation}
with associated operator
\begin{equation}\label{eq:probabilistic-continuous}
\mathbf R^{(jk)}f
=
\mathbf K^{(jk)}*f.
\end{equation}
The kernels $\mathbf K^{(jk)}$ satisfy the Calder\'on--Zygmund
conditions, the corresponding discrete analogues are precisely
$\mathcal R^{(jk)}$, and $\mathbf R^{(jk)}$ differs from the classical
$R^{(jk)}$ by convolution with an $L^1(\mathbb R^d)$ kernel.

The exact norm identity for the discrete Hilbert transform, the
dimension-free martingale estimates above, and the two integrable
perturbation relations together give a coherent reason to expect that
the four second-order models should retain the same sharp constants.
This led Ba\~nuelos and Kim \cite{BK2} to conjecture that the corresponding
$L^p$ and $\ell^p$ norms agree. In particular, for $j\neq k$ their
conjecture predicts
\[
\|\mathcal R^{(jk)}\|_{\ell^p\to\ell^p}
=
\|R_{\mathrm{dis}}^{(jk)}\|_{\ell^p\to\ell^p}
=
\|R^{(jk)}\|_{L^p\to L^p}.
\]

Against this background, the purpose of the present paper is to determine
the actual dimension dependence of the directly sampled second-order
operators. Our results show that direct sampling is substantially more
sensitive to the cancellation structure of the kernel than the
continuous and probabilistic theories suggest. Three distinct phenomena
occur: the off-diagonal transforms $R_{\mathrm{dis}}^{(jk)}$, $j\neq k$,
remain bounded but have rapidly growing norms; the diagonal transforms
$R_{\mathrm{dis}}^{(jj)}$ lose boundedness altogether; and the traceless
combinations $R_{\mathrm{dis}}^{(jj)}-R_{\mathrm{dis}}^{(kk)}$ recover
cancellation and are bounded throughout the strong $\ell^p$ scale, with
corresponding weak-type $(1,1)$ estimates.

We begin with the off-diagonal case. The following theorem gives the optimal dimension-dependent $\ell^2$ estimate.
\vskip0.3cm
\begin{theorem}\label{thm:offdiag-l2}
Let $d\geq2$ and let $j,k\in\{1,\ldots,d\}$ be distinct. For the second-order discrete Riesz transform $R_{\mathrm{dis}}^{(jk)}$ defined in \eqref{eq:direct-second-order}, we have
$$
c_d\left( \frac{2}{2^{\frac{d}{2}}}+\frac{\frac{8}{3}\left( d-2 \right)}{3^{\frac{d}{2}}}-\frac{\beta}{10^{\frac{d}{2}}}-\frac{\frac{48}{11}\left( d-2 \right)}{11^{\frac{d}{2}}} \right)  \leq \|R_{dis}^{\left( jk \right)}\|_{\ell ^2\left( \mathbb{Z}^d \right) \rightarrow \ell ^2\left( \mathbb{Z}^d \right)}\leq c_d\left( \frac{2}{2^{\frac{d}{2}}}+\frac{\alpha d}{3^{\frac{d}{2}}} \right)
$$
for some absolute constant $
\alpha >0,\beta >0
$.
In particular, when $d\to \infty$
$$
\|R_{dis}^{\left( jk \right)}\|_{\ell ^2\left( \mathbb{Z}^d \right) \rightarrow \ell ^2\left( \mathbb{Z}^d \right)}=c_d\left( \frac{2}{2^{\frac{d}{2}}}+\frac{\left( \frac{8}{3}+o\left( 1 \right) \right) d}{3^{\frac{d}{2}}} \right) .
$$
\end{theorem}
\begin{remark}
The proof yields an explicit, non-asymptotic upper bound for
$\|R_{\mathrm{dis}}^{(jk)}\|_{\ell^2(\mathbb Z^d)\to\ell^2(\mathbb Z^d)}$;
see \eqref{explicite upper bound}.  It is also useful to make the growth
implicit in Theorem~\ref{thm:offdiag-l2} explicit.  By Stirling's formula,
\begin{equation}\label{eq:stirling-cd}
c_d
=
\frac{\Gamma\left(\frac{d+2}{2}\right)}{\pi^{d/2}}
\sim
\sqrt{\pi d}
\left(\frac{d}{2\pi e}\right)^{d/2},
\qquad d\to\infty.
\end{equation}
Consequently,
\[
\|R_{\mathrm{dis}}^{(jk)}\|_{\ell^2\to\ell^2}
\sim
2\sqrt{\pi d}
\left(\frac{d}{4\pi e}\right)^{d/2}.
\]
Thus the directly sampled off-diagonal norm grows super-exponentially
with the dimension. This behavior is strikingly different from that of
the corresponding continuous and probabilistic second-order Riesz
transforms, whose sharp norms are dimension-free:
\[
\|2R^{(jk)}\|_{L^p\to L^p}
=
\|2\mathcal R^{(jk)}\|_{\ell^p\to\ell^p}
=
p^*-1.
\]
\end{remark}
\vskip0.1cm

We next turn to weak type $(1,1)$ and to the general $\ell^p$ theory. The following result gives the corresponding dimension-dependent bounds.
\vskip0.3cm
\begin{theorem}\label{thm:offdiag-lp}
Let $d\geq2$ and let $j,k\in\{1,\ldots,d\}$ be distinct. There exist absolute constants $\gamma,\lambda>0$ such that
$$\frac{2}{2^{\frac{d}{2}}}c_d\leq \|R_{dis}^{(jk)}\|_{\ell^1(\mathbb{Z}^d)\rightarrow \ell^{1,\infty}(\mathbb{Z}^d)}\leq  \left(\frac{2}{2^{\frac{d}{2}}}+\frac{\lambda d}{3^{\frac{d}{2}}}\right) c_d $$
and
$$c_d\left( \frac{2}{2^{\frac{d}{2}}}+\frac{\frac{8}{3}\left( d-2 \right)}{3^{\frac{d}{2}}}-\frac{\beta}{10^{\frac{d}{2}}}-\frac{\frac{48}{11}\left( d-2 \right)}{11^{\frac{d}{2}}} \right)\leq\|R_{dis}^{(jk)}\|_{\ell^p(\mathbb{Z}^d)\rightarrow \ell^p(\mathbb{Z}^d)}\leq \frac{p^*-1}{2}+c_d\left( \frac{2}{2^{\frac{d}{2}}}+\frac{\gamma d}{3^{\frac{d}{2}}} \right) ,$$
where $p^*=\max\left\{p,\frac{p}{p-1}\right\}$and $\beta$ is the absolute constant in Theorem \ref{thm:offdiag-l2}.

In particular, when $d\to \infty$
$$\frac{2}{2^{\frac{d}{2}}}c_d\leq \|R_{dis}^{(jk)}\|_{\ell^1(\mathbb{Z}^d)\rightarrow \ell^{1,\infty}(\mathbb{Z}^d)}\leq c_d\left(\frac{2}{2^{\frac{d}{2}}} +\frac{\left( \frac{8}{3}+o\left( 1 \right) \right) d}{3^{\frac{d}{2}}} \right)$$
and for every fixed $
1<p<\infty 
$,
$$\|R_{dis}^{\left( jk \right)}\|_{\ell ^p\left( \mathbb{Z}^d \right) \rightarrow \ell ^p\left( \mathbb{Z}^d \right)}=c_d\left(\frac{2}{2^{\frac{d}{2}}} +\frac{\left( \frac{8}{3}+o\left( 1 \right) \right) d}{3^{\frac{d}{2}}} \right) .$$

\end{theorem}
\vskip0.3cm
\begin{remark}
The sharp fixed-$p$ asymptotic in Theorem~\ref{thm:offdiag-lp} cannot be recovered by applying the Marcinkiewicz interpolation theorem only to the $\ell^2$ and weak-type $(1,1)$ estimates. Such an interpolation argument does not retain the precise second-order contribution at the scale $c_d d/3^{d/2}$. 
\end{remark}
\vskip0.1cm

The continuous diagonal theory provides a particularly striking point of
comparison. As recalled in \eqref{eq:continuous-diagonal-norm},
$R^{(jj)}$ is bounded on every $L^p(\mathbb R^d)$, $1<p<\infty$, with
the dimension-free sharp norm $\gamma(p)$; as a Calder\'on--Zygmund
operator, it is also of weak type $(1,1)$.

Unexpectedly, the directly sampled diagonal transforms behave in a
fundamentally different way. The positivity of the sampled kernel
destroys the cancellation present in the continuous principal-value
operator. Consequently, one not only loses any hope of dimension-free
control, but boundedness itself breaks down: $R_{\mathrm{dis}}^{(jj)}$
is unbounded on $\ell^p(\mathbb Z^d)$ for every $1<p<\infty$, and it
also fails to be of weak type $(1,1)$. More precisely, we have the
following theorem.
\vskip0.1cm
\begin{theorem}\label{thm:diagonal}
Let $d\geq2$ and $j\in\{1,\ldots,d\}$. Then
$R_{\mathrm{dis}}^{(jj)}$ is not bounded on $\ell^p(\mathbb Z^d)$ for
any $1<p<\infty$. Moreover, it is not bounded from
$\ell^1(\mathbb Z^d)$ to $\ell^{1,\infty}(\mathbb Z^d)$.
\end{theorem}
\vskip0.1cm
\begin{remark}
Theorem~\ref{thm:diagonal} shows that the obstruction in the diagonal
case is qualitative rather than merely dimensional.  In the continuous
Fourier-multiplier definition, the cancellation needed for
$R^{(jj)}$ is encoded in the singular-integral interpretation of the
operator; after direct sampling, the kernel displayed above is
nonnegative, so this cancellation disappears.  This also clarifies the
trichotomy in our results: the off-diagonal kernel $m_jm_k$ and the
traceless kernel $m_j^2-m_k^2$ retain cancellation, whereas the pure
diagonal kernel $m_j^2$ does not.  The logarithmic accumulation of its
positive lattice mass is precisely the mechanism behind both failures
in Theorem~\ref{thm:diagonal}.
\end{remark}
\vskip0.2cm

There is, however, a third regime. For $j\neq k$, consider the
traceless combination
\[
R_{\mathrm{dis}}^{(jj-kk)}
:=
R_{\mathrm{dis}}^{(jj)}
-
R_{\mathrm{dis}}^{(kk)}.
\]
The corresponding kernel
\[
c_d\frac{m_j^2-m_k^2}{|m|^{d+2}}
\]
recovers cancellation, and boundedness is restored. We first record the sharp $\ell^2$ estimate.
\vskip0.3cm
\begin{theorem}\label{thm:diagonal-difference}
Let $d\geq2$ and let $j,k\in\{1,\ldots,d\}$ be distinct. For $R_{\mathrm{dis}}^{(jj-kk)}:=R_{\mathrm{dis}}^{(jj)}-R_{\mathrm{dis}}^{(kk)}$, we have
\[
4c_d\left(1+\frac{d-2}{2^{d/2}}-\frac{\rho}{4^{d/2}}\right)
\leq
\|R_{\mathrm{dis}}^{(jj-kk)}\|_{\ell^2(\mathbb Z^d)\to\ell^2(\mathbb Z^d)}
\leq
4c_d\left(1+\frac{\mu d}{2^{d/2}}\right)
\]
for some absolute constants $\rho,\mu>0$. In particular, as $d\to\infty$,
\[
\|R_{\mathrm{dis}}^{(jj-kk)}\|_{\ell^2\to\ell^2}
=
4c_d\left(1+(1+o(1))\frac{d}{2^{d/2}}\right).
\]
\end{theorem}
\vskip0.2cm
The same asymptotic persists throughout the strong $\ell^p$ scale, while the weak-type norm admits matching dimension-dependent bounds.
\vskip0.2cm
\begin{theorem}\label{thm:diagonal-difference-lp}
Let $d\geq2$ and let $j,k\in\{1,\ldots,d\}$ be distinct. There exist absolute constants $\gamma,\lambda>0$ such that
\[
4c_d
\leq
\|R_{\mathrm{dis}}^{(jj-kk)}\|_{\ell^1(\mathbb Z^d)\to\ell^{1,\infty}(\mathbb Z^d)}
\leq
4c_d\left(1+\frac{\lambda d}{2^{d/2}}\right),
\]
and, for every $1<p<\infty$,
\[
4c_d\left(1+\frac{d-2}{2^{d/2}}-\frac{\rho}{4^{d/2}}\right)
\leq
\|R_{\mathrm{dis}}^{(jj-kk)}\|_{\ell^p(\mathbb Z^d)\to\ell^p(\mathbb Z^d)}
\leq
(p^*-1)+4c_d\left(1+\frac{\gamma d}{2^{d/2}}\right),
\]
where $p^*=\max\{p,p/(p-1)\}$ and $\rho$ is the absolute constant in Theorem~\ref{thm:diagonal-difference}.

In particular, as $d\to\infty$,
\[
4c_d
\leq
\|R_{\mathrm{dis}}^{(jj-kk)}\|_{\ell^1\to\ell^{1,\infty}}
\leq
4c_d\left[1+(1+o(1))\frac{d}{2^{d/2}}\right],
\]
and for every fixed $1<p<\infty$,
\[
\|R_{\mathrm{dis}}^{(jj-kk)}\|_{\ell^p\to\ell^p}
=
4c_d\left[1+(1+o(1))\frac{d}{2^{d/2}}\right].
\]
\end{theorem}
\vskip0.2cm
The preceding results reveal three distinct effects of direct
discretization. In the off-diagonal case, cancellation survives but
does not prevent a super-exponential dependence on dimension. In the
diagonal case, the loss of cancellation produces a logarithmic
accumulation of lattice mass and destroys boundedness. For the
difference of two diagonal components, cancellation is restored and
the Fourier multiplier method becomes available again. This
trichotomy is one of the main features of the second-order discrete
theory.

We conclude the introduction with a brief description of the ideas
used in the proofs. For $j\neq k$, by symmetry it suffices to study
$R_{\mathrm{dis}}^{(12)}$. We first introduce the associated
continuous--discrete operator
\[
\widetilde R_{\mathrm{dis}}^{(12)}F(x)
=
\sum_{n\in\mathbb Z^d\setminus\{0\}}
c_d\frac{n_1n_2}{|n|^{d+2}}F(x-n),
\qquad x\in\mathbb R^d.
\]
A transference result implies that
$\widetilde R_{\mathrm{dis}}^{(12)}$ and
$R_{\mathrm{dis}}^{(12)}$ have the same operator norm. In the
$\ell^2$ case, the problem is therefore reduced to estimating the
$L^\infty$ norm of the periodic Fourier multiplier
\[
m_{12}(\xi)
=
\operatorname{p.v.}
\sum_{m\in\mathbb Z^d\setminus\{0\}}
c_d\frac{m_1m_2}{|m|^{d+2}}
e^{-2\pi i m\cdot\xi}.
\]
The basic identity
\begin{equation}\label{eq:gamma-representation}
\frac1{|x|^{d+2}}
=
\frac{1}{\Gamma\left(\frac{d+2}{2}\right)}
\int_0^\infty
t^{d/2}e^{-t|x|^2}\,dt
\end{equation}
allows us to convert the multiplier into Gaussian lattice sums. By
parity, the resulting $d$-dimensional sums factor into
one-dimensional theta-type sums. Poisson summation gives a dual
Gaussian representation which is particularly effective on
$0<t\leq1$, while the original lattice expansion is more useful for
$t\geq1$. The contribution of the first few lattice shells then
determines the leading terms in Theorem~\ref{thm:offdiag-l2}.

For the weak-type $(1,1)$ and general $\ell^p$ estimates, we compare
$\widetilde R_{\mathrm{dis}}^{(12)}$ with the truncated continuous
second-order Riesz transform
\[
R_1^{(12)}F(x)
=
c_d
\int_{|y|\geq1}
\frac{y_1y_2}{|y|^{d+2}}F(x-y)\,dy.
\]
Their difference is a convolution operator. The main difficulty is
to obtain a sufficiently sharp $\ell^1$ estimate for the associated
error kernel. This is achieved by combining
\eqref{eq:gamma-representation} with a careful decomposition of the
lattice into near and far regions. These estimates yield
Theorem~\ref{thm:offdiag-lp}.

The diagonal transforms require a different mechanism. We use the
positivity of the kernel and consider the cone
\[
\Omega_j
=
\left\{
m\in\mathbb Z^d:
|m_j|\geq\frac{|m|}{\sqrt d}
\right\}.
\]
On this set,
\[
\frac{m_j^2}{|m|^{d+2}}
\geq
\frac1d\frac1{|m|^d},
\]
and the lattice mass satisfies a logarithmic lower bound of the form
\[
\sum_{\substack{m\in\Omega_j\\0<|m|\leq R}}
\frac1{|m|^d}
\gtrsim_d \log R.
\]
Testing the operator on normalized characteristic functions of large
cubes then gives logarithmic growth and proves both assertions of
Theorem~\ref{thm:diagonal}.

Finally, for
$R_{\mathrm{dis}}^{(jj-kk)}$ the kernel has cancellation again.
Its Fourier multiplier is
\[
m_{jj-kk}(\xi)
=
\operatorname{p.v.}
\sum_{m\in\mathbb Z^d\setminus\{0\}}
c_d
\frac{m_j^2-m_k^2}{|m|^{d+2}}
e^{-2\pi i m\cdot\xi}.
\]
We apply the same Gaussian representation and Poisson summation
scheme as in the off-diagonal case. The upper bound follows from a
small-time/large-time decomposition, while the lower bound is obtained
by evaluating the multiplier at a suitable half-period point; after
taking $j=1$ and $k=2$, one may use
\[
\xi_0=\left(0,\frac12,0,\ldots,0\right).
\]
This yields Theorem~\ref{thm:diagonal-difference}. To obtain the weak-type $(1,1)$ and general $\ell^p$ estimates, we compare the continuous--discrete operator with the truncated continuous transform associated with the harmonic polynomial $x_j^2-x_k^2$. The resulting error kernel is controlled by the same near/far decomposition as in the off-diagonal case, while its leading terms come from the first two lattice shells. This gives Theorem~\ref{thm:diagonal-difference-lp}.

The paper is organized as follows. In Section~2 we prove the optimal
dimension-dependent $\ell^2$ estimate for
$R_{\mathrm{dis}}^{(jk)}$, $j\neq k$. Section~3 is devoted to the
weak-type $(1,1)$ and $\ell^p$ estimates in the off-diagonal case.
Section~4 treats the diagonal transforms $R_{\mathrm{dis}}^{(jj)}$ and
proves both their $\ell^p$ unboundedness and their failure of weak type
$(1,1)$. In Section~5 we determine the optimal $\ell^2$ asymptotic for
the traceless combinations
$R_{\mathrm{dis}}^{(jj)}-R_{\mathrm{dis}}^{(kk)}$, and Section~6
establishes the corresponding weak-type $(1,1)$ and $\ell^p$ estimates.
The Fourier multiplier identities and Poisson summation formulas needed
in the proofs are collected in the appendix.
\section{The $\ell^2$ estimate for $j\ne k$}
In this section we prove the optimal dimension-dependent $\ell^2$ estimate for the off-diagonal second-order discrete Riesz transforms $R_{\mathrm{dis}}^{(jk)}$. By symmetry, it is enough to consider $j=1$ and $k=2$.

Define the continuous-discrete operator $\tilde{R}_{\mathrm{dis}}^{(12)}$
$$\tilde{R}_{\text{dis}}^{(12)}(F)(x)=\sum_{n \in \mathbb{Z}^d \setminus \{0\}}K_{12}(n)F(x-n),F \in L^p(\mathbb{R}^d),x \in \mathbb{R}^d,$$
where $K_{12}(n)=c_d\frac{n_1n_2}{|n|^{d+2}}$. A transference result of Ba\~{n}uelos, Kim and Kwa\'{s}nicki \cite{BKK}, stated below in the form needed here, identifies the operator norms of $R_{\mathrm{dis}}^{(12)}$ and $\widetilde{R}_{\mathrm{dis}}^{(12)}$.

\begin{lemma}[\cite{BKK}]\label{lem:transference}
    For the discrete Riesz transform $R_{\mathrm{dis}}^{(12)}$ and $1 < p < \infty$, the following identity holds
    \begin{equation*}
        \|R_{\mathrm{dis}}^{(12)}\|_{\ell^p(\mathbb{Z}^d) \to \ell^p(\mathbb{Z}^d)} = \|\tilde{R}_{\mathrm{dis}}^{(12)}\|_{L^p(\mathbb{R}^d) \to L^p(\mathbb{R}^d)}.
    \end{equation*}
\end{lemma}

By Lemma~\ref{lem:transference}, it suffices to estimate the $L^2$ norm of the continuous--discrete operator $\widetilde{R}_{\mathrm{dis}}^{(12)}$. By the standard multiplier criterion for convolution operators (see \cite{StWe}) and Lemma~\ref{lem:multiplier} in the Appendix, its Fourier multiplier is
    \begin{equation*}
        m_{12}(\xi) = \mathrm{p.v.} \sum_{m \in \mathbb{Z}^d \setminus \{0\}} K_{12}(m)e^{-2\pi i m \cdot \xi}=\mathrm{p.v.} \sum_{m \in \mathbb{Z}^d \setminus \{0\}}c_d\frac{m_1m_2}{|m|^{d+2}}e^{-2\pi i m \cdot \xi},~~~\xi \in [0, 1]^d.
    \end{equation*}
For comparison, in the one-dimensional discrete Hilbert-transform case, the corresponding multiplier is
$$-\mathrm{p.v.}\sum\limits_{n \in \mathbb{Z} \setminus \{0\}}\frac{2i}{\pi n}\sin(2\pi n \xi),$$
which can be written explicitly as $-i\frac{\xi}{|\xi|}(1-2|\xi|)$ and has $L^\infty$ norm $1$. In higher dimensions no comparably simple closed form is available for $m_{12}$, so a direct estimate of its $L^\infty$ norm is substantially more delicate.

To estimate $\|m_{12}\|_{L^\infty}$ from above and below, we first rewrite the multiplier in a form adapted to Gaussian summation.
Using the fact that $K_{12}(-m)=K_{12}(m)$ and identity (\ref{eq:gamma-representation}), we obtain
\begin{equation*}
 \begin{aligned}
    m_{12}(\xi) &=  \sum_{m \in \mathbb{Z}^d \setminus \{0\}} K_{12}(m)\cos(2\pi m \cdot \xi).\\
    &=  \sum_{m \in \mathbb{Z}^d \setminus \{0\}} c_d \frac{1}{\Gamma\left(\frac{d+2}{2}\right)} \left( \int_0^{\infty} t^{\frac{d}{2}} m_1m_2 e^{-t|m|^2} \mathrm{d}t \right) \cos(2\pi m \cdot \xi) \\
    &= \frac{c_d}{\Gamma \left( \frac{d+2}{2} \right)} \int_0^{\infty} t^{\frac{d}{2}} \sum_{m \in \mathbb{Z}^d \setminus \{0\}} m_1m_2 e^{-t|m|^2} \cos(2\pi m \cdot \xi) \mathrm{d}t,
\end{aligned}
\end{equation*}
where the interchange of the sum and integral is justified by Lemma~\ref{lem:gaussian-representation} in the Appendix.
Set
\begin{equation*}
\begin{aligned}
    S_1(t, x) &= \sum_{n \neq 0} n e^{-tn^2} \sin(2\pi n x), \\
    S_2(t, x) &= \sum_{n \in \mathbb{Z}} e^{-tn^2} \cos(2\pi n x),
\end{aligned}
\end{equation*}
and
$$\displaystyle S(t, \xi) = \sum_{m \in \mathbb{Z}^d \setminus \{0\}} m_1m_2 e^{-t|m|^2} \cos(2\pi m \cdot \xi).$$
Using the sum-to-product formulas of trigonometric functions and the fact
\begin{equation*}
\begin{cases}
    e^{-t|m|^2} = e^{-t m_1^2}e^{-t m_2^2}\cdots e^{-t m_d^2}, \\
    \cos(2\pi m \cdot \xi) = \cos(2\pi m_1\xi_1 + 2\pi m_2\xi_2 + \cdots + 2\pi m_d\xi_d),
\end{cases}
\end{equation*}
and the parity properties of the sine and cosine functions, we have
\begin{equation*}
\begin{aligned}
    S(t, \xi) &= -\left( \sum_{m_1 \neq 0} m_1 e^{-tm_1^2} \sin(2\pi m_1 \xi_1) \right) \left( \sum_{m_2 \neq 0} m_2e^{-tm_2^2} \sin(2\pi m_2 \xi_2) \right) \\
    &\quad  \cdots  \left( \sum_{m_d \in \mathbb{Z}} e^{-tm_d^2} \cos(2\pi m_d \xi_d) \right).
\end{aligned}
\end{equation*}
Consequently, the multiplier admits the representation
\begin{equation*}
    m_{12}(\xi) =  \frac{c_d}{\Gamma \left( \frac{d+2}{2} \right)} \int_0^{\infty} t^{\frac{d}{2}} S(t, \xi) \, \mathrm{d}t.
\end{equation*}
where
\begin{equation*}
    S(t, \xi) = -S_1(t, \xi_1)S_1(t, \xi_2) S_2(t, \xi_3)\cdots S_2(t, \xi_d).
\end{equation*}
\subsection{Upper bound}
We first establish the upper bound for $\|m_{12}\|_{L^\infty([0,1]^d)}$. Splitting the integral at $t=1$, we write
    \begin{equation*}
    \begin{aligned}
        |m_{12}(\xi)| &= \frac{c_d}{\Gamma \left( \frac{d+2}{2} \right)} \left| \int_0^{\infty} t^{\frac{d}{2}} S(t, \xi) \, \mathrm{d}t \right| \\
        &\le \frac{c_d}{\Gamma \left( \frac{d+2}{2} \right)} \left( \int_0^1 t^{\frac{d}{2}} |S(t, \xi)| \, \mathrm{d}t + \int_1^{\infty} t^{\frac{d}{2}} |S(t, \xi)| \, \mathrm{d}t \right),
    \end{aligned}
    \end{equation*}
and estimate the two ranges $0<t\leq1$ and $t\geq1$ separately.

For $0<t\leq1$, we use the same one-dimensional Poisson-summation estimates that appear in the first-order argument \cite[Section~2]{STX}. We recall the short derivation, since the two differentiated theta factors are used repeatedly below. Put
\[
\rho(x)=\operatorname{dist}(x,\mathbb Z)\in[0,1/2].
\]
By Lemma~\ref{lem:poisson-sums},
\[
S_2(t,x)=\sqrt{\frac\pi t}\sum_{k\in\mathbb Z}e^{-\pi^2(x-k)^2/t}.
\]
If $k_0$ is a nearest integer to $x$, then $|x-k|\ge |k-k_0|/2$ for $k\ne k_0$. Hence
\[
\sum_{k\ne k_0}e^{-\pi^2|x-k|^2/t}
\le 2\sum_{\ell\ge1}e^{-\pi^2\ell^2/(4t)}
\le 2\left(1+\frac{2}{\pi^2}\right)e^{-\pi^2/(4t)},
\]
and therefore
\begin{equation}\label{est of S2}
|S_2(t,x)|\le A t^{-1/2}e^{-\pi^2\rho(x)^2/t},
\qquad A=\sqrt\pi\left(3+\frac4{\pi^2}\right).
\end{equation}
Similarly,
\[
S_1(t,x)=\frac{\pi^{3/2}}{t^{3/2}}
\sum_{k\in\mathbb Z}(x-k)e^{-\pi^2(x-k)^2/t}.
\]
Writing $g(u)=u e^{-\pi^2u^2/t}$, the elementary bound
\[
|g'(u)|\le \left(1+\frac{2\pi^2u^2}{t}\right)e^{-\pi^2u^2/t}
\le 2e^{-\pi^2u^2/(2t)}
\]
and the mean value theorem give, exactly as in \cite{STX},
\begin{equation}\label{est of S1}
|S_1(t,x)|\le B t^{-3/2}\rho(x)e^{-\pi^2\rho(x)^2/(2t)},
\qquad B=9\pi^{3/2}.
\end{equation}
If $\rho_j=\rho(\xi_j)$, then \eqref{est of S2}--\eqref{est of S1} and the product formula for $S(t,\xi)$ yield
\begin{equation}\label{est I small}
\begin{aligned}
I_{\mathrm{small}}
&:=\frac{c_d}{\Gamma((d+2)/2)}\int_0^1t^{d/2}|S(t,\xi)|\,dt \\
&\le \frac{c_dB^2A^{d-2}}{\Gamma((d+2)/2)}\rho_1\rho_2
\int_0^1t^{-2}e^{-\pi^2(\rho_1^2+\cdots+\rho_d^2)/(2t)}\,dt \\
&\le c_d\frac{B^2A^{d-2}}{\pi^2\Gamma((d+2)/2)}.
\end{aligned}
\end{equation}
This is lower order compared with the first two large-time lattice contributions.

It remains to treat the large-time regime $t\geq1$.  By separating the
terms $|m|=1$ and estimating the remaining tails by an integral, as in
\cite[Section~2]{STX}, we have
\[
 |S_1(t,x)|\leq2e^{-t}+5e^{-4t},\qquad
 |S_2(t,x)|\leq1+2e^{-t}+3e^{-4t}\leq1+3e^{-t}.
\]

Using the fundamental inequality

\[
(1+x)^{d-2}
\le
1+(d-2)x+\frac{d^2}{2}(1+x)^{d-4} x^2,~~\text{for}~~x>0
\]
and the above estimates, we have
\begin{align*}
    \left|S(t,\xi)\right|&=\left|S_1(t, \xi_1)S_2(t, \xi_2) \cdots S_2(t, \xi_d)\right|\\
    &\le \bigl(2e^{-t}+5e^{-4t}\bigr)^2
\bigl(1+2e^{-t}+3e^{-4t}\bigr)^{d-2}\\
    &\le
4e^{-2t}\left( 1+\left( d-2 \right) \left( 2e^{-t}+3e^{-4t} \right) +\frac{d^2}{2}\left( 1+3e^{-t} \right) ^{d-4}\left( 3e^{-t} \right) ^2 \right) \\
&
\quad+22e^{-5t}\left( 1+2e^{-t}+3e^{-4t} \right) ^{d-2}\\
&\le 4e^{-2t}+8\left( d-2 \right) e^{-3t}+12\left( d-2 \right) e^{-6t}\\
&
\quad+18d^2e^{-4t}\left( 1+3e^{-t} \right) ^{d-4}+22e^{-5t}\left( 1+3e^{-t} \right) ^{d-2}
\end{align*}
A direct calculation gives that
\begin{align}\label{I large}\nonumber
    I_{\mathrm{large}} &= \frac{c_d}{\Gamma \left( \frac{d+2}{2} \right)} \int_1^{\infty} t^{\frac{d}{2}} |S(t, \xi)| \, \mathrm{d}t\\
    &\le c_d\left( \frac{2}{2^{\frac{d}{2}}}+\frac{\frac{8}{3}\left( d-2 \right)}{3^{\frac{d}{2}}}+\frac{2\left( d-2 \right)}{6^{\frac{d}{2}}} \right)+c_d R_1(d)+c_d R_2(d),
\end{align}
where
\[
R_1\left( d \right) :=\frac{18d^2}{\Gamma \left( \frac{d+2}{2} \right)}\int_1^{\infty}{t}^{\frac{d}{2}}e^{-4t}\bigl( 1+3e^{-t} \bigr) ^{d-4}\,dt,\]
and
\[
 R_2\left( d \right) :=\frac{22}{\Gamma \left( \frac{d+2}{2} \right)}\int_1^{\infty}{t}^{\frac{d}{2}}e^{-5t}\bigl( 1+3e^{-t} \bigr) ^{d-2}\,dt.
\]
It remains to control the two remainder terms. Splitting the integral into two parts and using the fact $(1+3e^{-t})^{d-4}\leq (e^{3e^{-t}})^{d-4}\leq e^{3de^{-t}} \leq e^{\frac{3}{d}}$ for $t \in [2\ln d, \infty)$, we have
\[
\begin{aligned}
\int_1^\infty
t^{\frac{d}{2}}e^{-4t}
\bigl(1+3e^{-t}\bigr)^{d-4}\,dt
&=
\int_1^{2\ln d}
t^{\frac{d}{2}}e^{-4t}
\bigl(1+3e^{-t}\bigr)^{d-4}\,dt   \\
&\quad
+
\int_{2\ln d}^{\infty}
t^{\frac{d}{2}}e^{-4t}
\bigl(1+3e^{-t}\bigr)^{d-4}\,dt  \\
&\le
(2\ln d)^{\frac{d+2}{2}}
\left(1+\frac3e\right)^d
+
e^{\frac{3}{d}}\frac{\Gamma\left(\frac{d}{2}+1\right)}
{4^{\frac{d}{2}+1}}.
\end{aligned}
\]
Similarly,
\[
\begin{aligned}
\int_1^\infty
t^{\frac{d}{2}}e^{-5t}
\bigl(1+3e^{-t}\bigr)^{d-2}\,dt
&\le
(2\ln d)^{\frac{d+2}{2}}
\left(1+\frac3e\right)^d
+
e^{\frac{3}{d}}\frac{\Gamma\left(\frac{d}{2}+1\right)}
{5^{\frac{d}{2}+1}}.
\end{aligned}
\]
Therefore,
\begin{align}\label{R1}
 R_1\left( d \right) \le \frac{18d^2\left( 2\ln d \right) ^{\frac{d+2}{2}}\left( 1+\frac{3}{e} \right) ^d}{\Gamma \left( \frac{d+2}{2} \right)}+18d^2e^{\frac{3}{d}}\frac{1}{4^{\frac{d}{2}+1}},
\end{align}

and
\begin{align}\label{R2}
R_2\left( d \right) \le \frac{22\left( 2\ln d \right) ^{\frac{d+2}{2}}\left( 1+\frac{3}{e} \right) ^d}{\Gamma \left( \frac{d+2}{2} \right)}+22e^{\frac{3}{d}}\frac{1}{5^{\frac{d}{2}+1}}.
\end{align}

Combining \eqref{est I small}, \eqref{I large}, \eqref{R1}, and \eqref{R2}, we obtain
\begin{align}\label{explicite upper bound}\nonumber
    ||m_{12}(\xi)||_{L^\infty}
&\le
c_d\left( \frac{2}{2^{\frac{d}{2}}}+\frac{\frac{8}{3}\left( d-2 \right)}{3^{\frac{d}{2}}}+\frac{2\left( d-2 \right)}{6^{\frac{d}{2}}}+18d^2e^{\frac{3}{d}}\frac{1}{4^{\frac{d}{2}+1}}+22e^{\frac{3}{d}}\frac{1}{5^{\frac{d}{2}+1}} \right)\\
&\quad +c_d\left( \frac{18d^2\left( 2\ln d \right) ^{\frac{d+2}{2}}\left( 1+\frac{3}{e} \right) ^d}{\Gamma \left( \frac{d+2}{2} \right)}+\frac{22\left( 2\ln d \right) ^{\frac{d+2}{2}}\left( 1+\frac{3}{e} \right) ^d}{\Gamma \left( \frac{d+2}{2} \right)}+\frac{B^2A^{d-2}}{\Gamma \left( \frac{d+2}{2} \right) \pi ^2} \right) ,
\end{align}
which implies
$$\|R_{\mathrm{dis}}^{(12)}\|_{\ell^2(\mathbb{Z}^d) \to \ell^2(\mathbb{Z}^d)} \le c_d\left( \frac{2}{2^{\frac{d}{2}}}+\frac{\left( \frac{8}{3}+o\left( 1 \right) \right) d}{3^{\frac{d}{2}}} \right) .$$

\subsection{Lower bound}
We now establish the matching lower bound. The elementary estimate from the Introduction already gives a lower bound of order $c_d2^{-d/2}$ for $ \|R_{\mathrm{dis}}^{(jk)}\|_{\ell^2(\mathbb{Z}^d) \to \ell^2(\mathbb{Z}^d)}$ . To recover the second leading term, we evaluate the multiplier at a carefully chosen point.

Recall that
\begin{equation*}
    |m_{12}(\xi)| = \frac{c_d}{\Gamma \left( \frac{d+2}{2} \right)} \left| \int_0^{\infty} t^{\frac{d}{2}} S(t, \xi) \, \mathrm{d}t \right| ,
\end{equation*}
By continuity, we evaluate $m_{12}$ at $\xi_0 = \left(\frac{1}{4}, \frac{1}{4}, 0, \dots, 0\right) \in [0,1]^d$. Then
\begin{equation*}
    -S(t, \xi_0) = S_1\left(t, \frac{1}{4}\right)^2[S_2(t, 0)]^{d-2},
\end{equation*}
where
\begin{equation*}
    S_2\left( t,0 \right) =\sum_{m\in \mathbb{Z}}{e}^{-tm^2}=1+2e^{-t}+2e^{-4t}+\cdots \ge 1+2e^{-t},\quad \text{for\,\,all\,\,}t>0.
\end{equation*}
and
\begin{equation*}
    S_1\left(t, \frac{1}{4}\right) = \sum_{m \neq 0} m e^{-tm^2} \sin\left(\frac{\pi m}{2}\right)=2 \sum_{k=0}^{\infty} (-1)^k (2k+1)e^{-t(2k+1)^2}.
\end{equation*}
By Lemma~\ref{lem:poisson-sums},
$$S_1\left(t, \frac{1}{4}\right)=\frac{\pi}{4t} \sqrt{\frac{\pi}{t}} \sum_{k=0}^{\infty} (-1)^k (2k+1) e^{-\frac{\pi^2(2k+1)^2}{16t}}.$$
For $t\in(0,1]$, the sequence $b_k(t)=e^{-\frac{\pi^2(2k+1)^2}{16t}}$ is strictly decreasing in $k$ and tends to zero. Pairing consecutive terms therefore gives
$$S_1\left(t, \frac{1}{4}\right) = \frac{\pi}{4t}\sqrt{\frac{\pi}{t}}[(b_0(t) - b_1(t)) + (b_2(t) - b_3(t)) + \cdots] > 0~~\text{for}~~~t\in (0,1].$$

For $t\geq \frac{1}{50}$, one checks that $a_k(t)=(2k+1)e^{-t(2k+1)^2}$ is strictly decreasing in $k$ for $k\geq2$ and satisfies $a_k(t)\to0$ as $k\to\infty$.
Thus
\begin{equation*}
\begin{aligned}
    S_1\left(t, \frac{1}{4}\right) &> 2(a_0(t) - a_1(t)) = 2(e^{-t} - 3e^{-9t}), \quad t \ge \frac{1}{50}.
\end{aligned}
\end{equation*}
Moreover, when $t\ge \frac{\ln 3}{8}$, $S_1\left( t,\frac{1}{4} \right) \ge 2\left( e^{-t}-3e^{-9t} \right) \ge 0.$
Hence, for $t\geq \frac{\ln 3}{8}$,
$$
S_{1}^{2}\left( t,\frac{1}{4} \right) >4\left( e^{-t}-3e^{-9t} \right) ^2,
$$
which implies
 $$-S(t, \xi_0) \ge
4\left( e^{-2t}+2\left( d-2 \right) e^{-3t}-6e^{-10t}-12\left( d-2 \right) e^{-11t}+9e^{-18t}+18\left( d-2 \right) e^{-19t} \right).$$
Therefore
\[
\begin{aligned}
|m_{12}(\xi_0)|
&=
\frac{c_d}{\Gamma\left(\frac{d+2}{2}\right)}
\int_0^\infty t^{d/2}\bigl(-S(t,\xi_0)\bigr)\,dt \\
&\geq
\frac{c_d}{\Gamma\left(\frac{d+2}{2}\right)}
\left(
\int_0^\infty t^{d/2}\varphi(t)\,dt
-
\int_0^{\frac{\ln3}{8}} t^{d/2}\varphi(t)\,dt
\right),
\end{aligned}
\]
where
\begin{equation*}
    \varphi \left( t \right) \,:=4\left( e^{-2t}+2\left( d-2 \right) e^{-3t}-6e^{-10t}-12\left( d-2 \right) e^{-11t}+9e^{-18t}+18\left( d-2 \right) e^{-19t} \right).
\end{equation*}

The preceding estimates yield
$$
\begin{aligned}
	|m_{12}\left( \xi _0 \right) |\ge c_d\left( \frac{2}{2^{\frac{d}{2}}}+\frac{\frac{8}{3}\left( d-2 \right)}{3^{\frac{d}{2}}}-\frac{\beta}{10^{\frac{d}{2}}}-\frac{\frac{48}{11}\left( d-2 \right)}{11^{\frac{d}{2}}} \right) .\\
\end{aligned}
$$
Therefore
$$\|R_{dis}^{(jk)}\|_{\ell^2(\mathbb{Z}^d)\rightarrow \ell^2(\mathbb{Z}^d)}\geq c_d\left( \frac{2}{2^{\frac{d}{2}}}+\frac{\frac{8}{3}\left( d-2 \right)}{3^{\frac{d}{2}}}-\frac{\beta}{10^{\frac{d}{2}}}-\frac{\frac{48}{11}\left( d-2 \right)}{11^{\frac{d}{2}}} \right) ,$$
which completes the proof of Theorem~\ref{thm:offdiag-l2}.

\section{The $\ell^{1,\infty}$ and $\ell^p$ estimates for $j\ne k$}
We continue with $j=1$, $k=2$. The strong-type transference identity is Lemma~\ref{lem:transference}; at the endpoint we use the following analogue.
\begin{lemma}\label{lem:weak-transference}
\[
\|\widetilde R_{\mathrm{dis}}^{(12)}\|_{L^1(\mathbb R^d)\to L^{1,\infty}(\mathbb R^d)}
=
\|R_{\mathrm{dis}}^{(12)}\|_{\ell^1(\mathbb Z^d)\to\ell^{1,\infty}(\mathbb Z^d)}.
\]
\end{lemma}
\begin{proof}
This is the kernel-independent cube-transference argument of \cite[Lemma~3.1]{STX}; we record the two short steps. Let $Q=[-1/2,1/2)^d$. For $F\in L^1(\mathbb R^d)$ and $x\in Q$, define $F_x(n)=F(x+n)$. Then
\[
\widetilde R_{\mathrm{dis}}^{(12)}F(x+n)=R_{\mathrm{dis}}^{(12)}F_x(n),
\qquad n\in\mathbb Z^d.
\]
Thus, for every $\lambda>0$,
\[
\begin{aligned}
\lambda\bigl|\{y:|\widetilde R_{\mathrm{dis}}^{(12)}F(y)|>\lambda\}\bigr|
&=\lambda\int_Q\#\{n:|R_{\mathrm{dis}}^{(12)}F_x(n)|>\lambda\}\,dx \\
&\le \|R_{\mathrm{dis}}^{(12)}\|_{\ell^1\to\ell^{1,\infty}}
\int_Q\|F_x\|_{\ell^1}\,dx \\
&=\|R_{\mathrm{dis}}^{(12)}\|_{\ell^1\to\ell^{1,\infty}}\|F\|_{L^1},
\end{aligned}
\]
which gives one inequality.

Conversely, for $f\in\ell^1(\mathbb Z^d)$ set
\[
F(x)=\sum_{n\in\mathbb Z^d}f(n)\mathbf1_Q(x-n).
\]
Then $\|F\|_1=\|f\|_1$ and, for $x\in n+Q$,
\[
\widetilde R_{\mathrm{dis}}^{(12)}F(x)=R_{\mathrm{dis}}^{(12)}f(n).
\]
Hence the distribution functions, and therefore the weak $L^1$ quasi-norms, agree:
\[
\|\widetilde R_{\mathrm{dis}}^{(12)}F\|_{L^{1,\infty}}
=\|R_{\mathrm{dis}}^{(12)}f\|_{\ell^{1,\infty}}.
\]
Taking the supremum over $f$ gives the reverse inequality.
\end{proof}

We next isolate the lattice sum responsible for the sharp error term. It is worth noting that the result below may be of independent interest.
\begin{lemma}\label{lem:lattice-sum}
There exists an absolute constant $\delta>0$ such that
$$\sum_{\substack{z\in\mathbb Z^d\\1\le |z|\le d^2}}
    \frac{\left| z_1 \right|\left| z_2 \right|}{\left| z \right|^{d+2}}\leq \frac{2}{2^{\frac{d}{2}}}+\frac{\delta d}{3^{\frac{d}{2}}}.$$
Moreover, as $d\to\infty$,
$$\sum_{\substack{z\in\mathbb Z^d\\1\le |z|\le d^2}}
    \frac{\left| z_1 \right|\left| z_2 \right|}{\left| z \right|^{d+2}}= \frac{2}{2^{\frac{d}{2}}}+\frac{\left( \frac{8}{3}+o\left( 1 \right) \right) d}{3^{\frac{d}{2}}}.$$
\end{lemma}
\begin{proof}
Denote
$$B_d=\sum_{\substack{z\in\mathbb Z^d\\1\le |z|\le d^2}}
    \frac{\left| z_1 \right|\left| z_2 \right|}{\left| z \right|^{d+2}}.$$
Clearly,
$$B_d\geq \frac{2}{2^{\frac{d}{2}}}+\sum_{\substack{z_1=z_2=\pm 1\\ |z|= \sqrt{3}}}\frac{\left| z_1 \right|\left| z_2 \right|}{\left| z \right|^{d+2}}=\frac{2}{2^{\frac{d}{2}}}+\frac{\frac{8}{3}\left( d-2 \right)}{3^{\frac{d}{2}}}.$$ For the upper bound we follow the strategy developed in Section~2. Using the identity $\frac{1}{|z|^{d+2}}=\frac{1}{\Gamma\left(\frac{d+2}{2}\right)}
    \int_0^{\infty}t^{\frac{d}{2}}e^{-t|z|^2}\,dt $ we have
\begin{align*}
    B_d
     & =
   \sum_{\substack{z\in\mathbb Z^d\\1\le |z|\le d^2}}
    \frac{\left| z_1 \right|\left| z_2 \right|}{\left| z \right|^{d+2}}
    =
    \sum_{\substack{z\in\mathbb Z^d\\1\le |z|\le d^2}} |z_1||
    z_2|\frac{1}{\Gamma\left(\frac{d+2}{2}\right)}
    \int_0^{\infty}t^{\frac{d}{2}}e^{-t|z|^2}\,dt                     \\
    & \le
    \frac{1}{\Gamma\left(\frac{d+2}{2}\right)}
    \int_0^{\infty}t^{\frac{d}{2}}
    \left(\sum_{z_1=-d^2}^{d^2}|z_1|e^{-tz_1^2}\right)
    \left(\sum_{z_2=-d^2}^{d^2}|z_2|e^{-tz_2^2}\right)
    \left(\sum_{k=-\infty}^{\infty}e^{-tk^2}\right)^{d-2}\,dt           \\
    & =
    \frac{4}{\Gamma\left(\frac{d+2}{2}\right)}
    \int_0^{\infty}{t}^{\frac{d}{2}}D_{d}^{2}\left( t \right) \Theta \left( t \right) ^{d-2}\,dt,
\end{align*}
where
\[
    \Theta(t):=\sum_{k=-\infty}^{\infty}e^{-tk^2},
    \qquad
    D_d(t):=\sum_{k=1}^{d^2}ke^{-tk^2}.
\]
Set
\[
    J_{\mathrm{small}}
    =
    \frac{4}{\Gamma\left(\frac{d+2}{2}\right)}
    \int_0^{1}{t}^{\frac{d}{2}}D_{d}^{2}\left( t \right) \Theta \left( t \right) ^{d-2}\,dt,
\]
and
\[
    J_{\mathrm{large}}
    =
     \frac{4}{\Gamma\left(\frac{d+2}{2}\right)}
    \int_1^{\infty} t^{\frac{d}{2}}D_{d}^{2}\left( t \right) \Theta \left( t \right) ^{d-2}\,dt.
\]

For $t\in(0,1]$, the Poisson summation formula (see Lemma \ref{lem:poisson-sums}) reveals
\[
    \Theta(t)
    =\sqrt{\frac{\pi}{t}}
    \left(1+2\sum_{n=1}^{\infty}e^{-\frac{\pi^2n^2}{t}}\right)
    \le
    4\sqrt{\frac{\pi}{t}}
    \left(1+e^{-\frac{\pi^2}{t}}\right).
\]
Consequently,
\begin{align*}
    J_{\mathrm{small}}
    \le
    4\frac{(4\sqrt{\pi})^{d-2}}{\Gamma(\frac{d+2}{2})}
    \biggl(
    \int_0^1 tD_{d}^{2}(t)\,dt
    +
    \int_0^1
    t\left[
    \left(1+e^{-\frac{\pi^2}{t}}\right)^{d-2}-1
    \right]
    D_{d}^{2}(t)\,dt
    \biggr).
\end{align*}
For the first term,
\[
    \int_0^1tD_{d}^{2}(t)\,dt
    \le\frac{\left( d^2+1 \right) ^4}{4}\int_0^1{tdt}
    = \frac{\left( d^2+1 \right) ^4}{8}.
\]
For the second term, notice that for $t\in(0,1]$
\[
    D_d(t)\leq \sum_{a=1}^{\infty}ae^{-ta^2}
    \le \left(1+\frac1{2t}\right)e^{-t}
    \le \frac{2}{t}.
\]
hence
\begin{align*}
    &\int_0^1t
    \left[
    \left(1+e^{-\frac{\pi^2}{t}}\right)^{d-2}-1
    \right]D_{d}^2{}(t)\,dt
     \le
    4\int_0^1
    \left[
    \left(1+e^{-\frac{\pi^2}{t}}\right)^{d-2}-1
    \right]\frac{dt}{t}                                                 \\
    & =
    4\int_0^1
    \sum_{m=1}^{d-2}\binom{d-2}{m}e^{-\frac{m\pi^2}{t}}\frac{dt}{t}
    \le
    4\sum_{m=1}^{d-2}\binom{d-2}{m}e^{-m\pi^2}           \le 4(1+e^{-\pi^2})^{d-2}.
\end{align*}
Combining these estimates gives
\begin{align}\label{est J small}
    J_{\mathrm{small}}
    \le
    \frac{\left( 4\sqrt{\pi} \right) ^{d-2}\left( d^2+1 \right) ^4}{2\Gamma \left( \frac{d+2}{2} \right)}
    +
    16\frac{\left( 4\sqrt{\pi} \right) ^{d-2}}{\Gamma \left( \frac{d+2}{2} \right)}\left( 1+e^{-\pi ^2} \right) ^{d-2}.
\end{align}

For $J_{\mathrm{large}}$, since $ D_d(t)
    \le \sum\limits_{k=1}^{\infty}ke^{-tk^2}$ we may use the same large-time estimate as in Section~2 to obtain
 \begin{align}\label{est J large}\nonumber
    J_{large}
&\le
\frac{2}{2^{\frac{d}{2}}}+\frac{\frac{8}{3}\left( d-2 \right)}{3^{\frac{d}{2}}}+\frac{2\left( d-2 \right)}{6^{\frac{d}{2}}}+18d^2e^{\frac{3}{d}}\frac{1}{4^{\frac{d+2}{2}}}+22e^{\frac{3}{d}}\frac{1}{5^{\frac{d+2}{2}}}\\
&\quad +\frac{18d^2\left( 2\ln d \right) ^{\frac{d+2}{2}}\left( 1+\frac{3}{e} \right) ^d}{\Gamma \left( \frac{d+2}{2} \right)}+\frac{22\left( 2\ln d \right) ^{\frac{d+2}{2}}\left( 1+\frac{3}{e} \right) ^d}{\Gamma \left( \frac{d+2}{2} \right)}.
 \end{align}
By (\ref{est J small}) and (\ref{est J large}), we have
\begin{align}\label{est Bd}
    B_d=\frac{2}{2^{\frac{d}{2}}}+\frac{\left( \frac{8}{3}+o\left( 1 \right) \right) d}{3^{\frac{d}{2}}}.
\end{align}
which completes the proof of the lemma.
\end{proof}
For $z\in\mathbb Z^d$ put
\[
K_{12}^*(x)=c_d\frac{x_1x_2}{|x|^{d+2}}\mathbf1_{\{|x|\ge1\}},\qquad
\widetilde K_{12}^*(z)=\int_Q\!\int_Q\bigl(K_{12}^*(z+s-t)-K_{12}^*(z)\bigr)\,dt\,ds.
\]
\begin{lemma}\label{lem:error-kernel}
There is an absolute $\eta>0$ such that
\[
\|\widetilde K_{12}^*\|_{\ell^1}\le c_d\left(\frac2{2^{d/2}}+\frac{\eta d}{3^{d/2}}\right),
\]
and, as $d\to\infty$,
\[
\|\widetilde K_{12}^*\|_{\ell^1}\le c_d\left[\frac2{2^{d/2}}+\left(\frac83+o(1)\right)\frac{d}{3^{d/2}}\right].
\]
\end{lemma}
\begin{proof}
We use the near/far decomposition introduced for the first-order kernel in \cite[Section~3]{STX}. Differentiating
$K_{12}(x)=c_dx_1x_2|x|^{-d-2}$ gives $|\nabla K_{12}(x)|\le Cc_dd|x|^{-d-1}$. Since $|s-t|\le\sqrt d$ for $s,t\in Q$, the mean value theorem yields, whenever $|z|>d^2$,
\[
|\widetilde K_{12}^*(z)|\le Cc_dd^{3/2}|z|^{-d-1}.
\]
As in \cite[Section~3]{STX}, comparison of the disjoint unit cubes $z+Q$ with the exterior region $\{|x|>d^2/2\}$ gives
\begin{equation}\label{est far region}
\begin{aligned}
\sum_{|z|>d^2}|\widetilde K_{12}^*(z)|
&\lesssim c_dd^{3/2}\sum_{|z|>d^2}|z|^{-d-1} \\
&\lesssim c_dd^{3/2}\int_{|x|>d^2/2}|x|^{-d-1}\,dx
=o\left(c_d\frac d{3^{d/2}}\right).
\end{aligned}
\end{equation}
For $|z|\le d^2$, let $\Phi(u)=\prod_{r=1}^d(1-|u_r|)_+$. The change of variables $u=s-t$ gives
\[
\widetilde K_{12}^*(z)=\int_{[-1,1)^d}\Phi(u)
\bigl(K_{12}^*(z+u)-K_{12}^*(z)\bigr)\,du,
\]
and the partition of unity $\sum_{z\in\mathbb Z^d}\Phi(x-z)=1$ gives
\begin{equation}\label{est near region}
\sum_{|z|\le d^2}|\widetilde K_{12}^*(z)|\le I_d+II_d,
\end{equation}
where
\[
I_d=\sum_{|z|\le d^2}\int\Phi(u)|K_{12}^*(z+u)|\,du,
\qquad II_d=\sum_{|z|\le d^2}|K_{12}^*(z)|.
\]
If $\Phi(x-z)\ne0$ and $|z|\le d^2$, then $|x|\le d^2+\sqrt d$. Hence the partition of unity and polar coordinates give
\begin{equation}\label{est Id}
\begin{aligned}
I_d
&\le c_d\int_{1\le|x|\le d^2+\sqrt d}\frac{|x_1x_2|}{|x|^{d+2}}\,dx \\
&=c_d\log(d^2+\sqrt d)\int_{\mathbb S^{d-1}}|\theta_1\theta_2|\,d\sigma(\theta) \\
&=O\!\left(c_d\log d\,\frac{\pi^{(d+1)/2}}{\Gamma((d+1)/2)}\right)
=o\left(c_d\frac d{3^{d/2}}\right).
\end{aligned}
\end{equation}
Lemma~\ref{lem:lattice-sum} gives
\begin{equation}\label{est IId}
II_d\le c_d\left[\frac2{2^{d/2}}+\left(\frac83+o(1)\right)\frac d{3^{d/2}}\right].
\end{equation}
Equations \eqref{est far region}--\eqref{est IId} prove the asymptotic upper bound; the uniform estimate follows by enlarging the absolute constant for the remaining dimensions.
\end{proof}

\subsection{The $\ell^{1,\infty}$ case}

For the lower bound, take $f=\delta_0$. At the four points
$\pm e_1\pm e_2$,
\[
|R_{\mathrm{dis}}^{(12)}f(n)|
=
\frac{c_d}{2^{d/2+1}}.
\]
Taking
$\lambda=(1-\varepsilon)c_d/2^{d/2+1}$ and letting
$\varepsilon\rightarrow0$, we obtain
\[
\|R_{\mathrm{dis}}^{(12)}\|_{\ell^1\to\ell^{1,\infty}}
\geq
\frac{2c_d}{2^{d/2}}.
\]

We turn to the upper bound. For $f\in\ell^1(\mathbb Z^d)$, let
\[
F(x)
=
\sum_{n\in\mathbb Z^d}
f(n)\mathbf 1_Q(x-n),
\qquad
Q=[-1/2,1/2)^d.
\]
As in \cite[Section~3]{STX}, write
\[
\widetilde R_{\mathrm{dis}}^{(12)}F
=
R_1^{(12)}F+E.
\]
The same cube computation gives
\begin{equation}\label{eq:offdiag-error-L1}
\|E\|_{L^1(\mathbb R^d)}
\leq
\|\widetilde K_{12}^*\|_{\ell^1(\mathbb Z^d)}
\|f\|_{\ell^1(\mathbb Z^d)}.
\end{equation}

To combine the two terms without losing the sharp leading constant,
we use an asymmetric weak-type decomposition. Set
\[
\varepsilon_d=3^{-d/2}.
\]
For every $\lambda>0$,
\[
\left\{
\left|
\widetilde R_{\mathrm{dis}}^{(12)}F
\right|>\lambda
\right\}
\subset
\left\{
|E|>\frac{\lambda}{1+\varepsilon_d}
\right\}
\cup
\left\{
|R_1^{(12)}F|
>
\frac{\varepsilon_d\lambda}{1+\varepsilon_d}
\right\}.
\]
Consequently,
\begin{align}
\|\widetilde R_{\mathrm{dis}}^{(12)}F\|_{L^{1,\infty}}
&\leq
(1+\varepsilon_d)\|E\|_{L^1}
+
\frac{1+\varepsilon_d}{\varepsilon_d}
\|R_1^{(12)}F\|_{L^{1,\infty}}
\nonumber\\
&=
(1+3^{-d/2})\|E\|_{L^1}
+
(1+3^{-d/2})3^{d/2}
\|R_1^{(12)}F\|_{L^{1,\infty}}.
\label{eq:offdiag-weak-splitting}
\end{align}

On the other hand, Janakiraman's argument \cite{Ja} gives
\[
\|R_1^{(12)}F\|_{L^{1,\infty}(\mathbb R^d)}
\leq
C'(\log d)^2\|F\|_{L^1(\mathbb R^d)}
=
C'(\log d)^2\|f\|_{\ell^1(\mathbb Z^d)},
\]
where $C'>0$ is independent of $d$. Hence the second term in
\eqref{eq:offdiag-weak-splitting} is bounded by
\[
C'(1+3^{-d/2})3^{d/2}(\log d)^2
\|f\|_{\ell^1}.
\]
By Stirling's formula,
\[
3^{d/2}(\log d)^2
=
o\left(
c_d\frac{d}{3^{d/2}}
\right),
\qquad d\to\infty.
\]
Thus the contribution of the truncated continuous transform is
negligible at the scale of the second main term.

Combining \eqref{eq:offdiag-error-L1},
\eqref{eq:offdiag-weak-splitting}, and
Lemma~\ref{lem:error-kernel}, we obtain
\[
\|\widetilde R_{\mathrm{dis}}^{(12)}F\|_{L^{1,\infty}}
\leq
c_d
\left[
\frac{2}{2^{d/2}}
+
\left(\frac83+o(1)\right)
\frac{d}{3^{d/2}}
\right]
\|f\|_{\ell^1}.
\]
Finally, Lemma~\ref{lem:weak-transference} yields
\[
\|R_{\mathrm{dis}}^{(12)}\|_{\ell^1\to\ell^{1,\infty}}
\leq
c_d
\left[
\frac{2}{2^{d/2}}
+
\left(\frac83+o(1)\right)
\frac{d}{3^{d/2}}
\right].
\]
The non-asymptotic upper bound in
Theorem~\ref{thm:offdiag-lp} follows in the same way, after enlarging
the absolute constant to absorb the finitely many small dimensions.

\subsection{The $\ell^p$ case}
Because $K_{12}$ is real and even, $\widetilde R_{\mathrm{dis}}^{(12)}$ is self-adjoint. Thus duality and Riesz--Thorin imply
\[
\|R_{\mathrm{dis}}^{(12)}\|_{\ell^p\to\ell^p}\ge\|R_{\mathrm{dis}}^{(12)}\|_{\ell^2\to\ell^2},
\]
which, together with Theorem~\ref{thm:offdiag-l2}, gives the lower bound in Theorem~\ref{thm:offdiag-lp}.

For the upper bound let $R_1^{(12)}$ be convolution with $K_{12}^*$. Since $x_1x_2$ is a homogeneous harmonic polynomial of degree two, the higher-order truncation factorization of Kucharski--Kwa\'{s}nicki--Wr\'{o}bel \cite{KKW}, together with dilation and strong convergence of the truncations, gives
\begin{equation}\label{eq:truncated-second-order-norm}
\|R_1^{(12)}\|_{L^p\to L^p}=\|R^{(12)}\|_{L^p\to L^p}=\frac{p^*-1}{2}.
\end{equation}
Proposition~6.1 of \cite{BKK} and Lemma~\ref{lem:error-kernel} now give
\begin{equation}\label{eq:lp-upper}
\|R_{\mathrm{dis}}^{(jk)}\|_{\ell^p\to\ell^p}
\le\frac{p^*-1}{2}+c_d\left[\frac2{2^{d/2}}+\left(\frac83+o(1)\right)\frac d{3^{d/2}}\right].
\end{equation}
For fixed $p$, $(p^*-1)/2=o(c_dd/3^{d/2})$, and the lower bound from Theorem~\ref{thm:offdiag-l2} together with \eqref{eq:lp-upper} proves the stated sharp asymptotic.

\section{The diagonal case $j=k$}
In this section we prove both assertions of Theorem~\ref{thm:diagonal}. By symmetry, it is enough to consider $j=k=1$. Once unboundedness on $\ell^2$ is established, self-adjointness and interpolation rule out boundedness on every $\ell^p$, $1<p<\infty$. The argument is based on positivity rather than Fourier multipliers.

The operator is
\[
R_{\mathrm{dis}}^{(11)}f(n)=\sum_{m\in\mathbb Z^d\setminus\{0\}}K_{11}(m)f(n-m),
\qquad K_{11}(m)=c_d\frac{m_1^2}{|m|^{d+2}},
\]
and we exploit the mass of this nonnegative kernel on the cone
\[
\Omega=\left\{y\in\mathbb Z^d:|y_1|\ge \frac{|y|}{\sqrt d}\right\}.
\]

\begin{lemma}\label{lem:cone}
For every $d\ge2$ there exist constants $u_d,v_d>0$, depending only on $d$, such that
\[
\sum_{\substack{y\in\Omega\\0<|y|\le R}}|y|^{-d}\ge u_d\log R,
\qquad R\ge v_d.
\]
\end{lemma}

\begin{proof}
It is enough to work in $\Omega_+=\{y:y_1\ge |y|/\sqrt d\}$. Put
$\alpha=(2\sqrt{d-1})^{-1}$ and, for $k\ge0$, define
\[
A_k=\left\{y\in\mathbb Z^d:2^k\le y_1<2^{k+1},\ |y_j|\le\lfloor\alpha2^k\rfloor,
\ 2\le j\le d\right\}.
\]
If $y\in A_k$, then $|y_j|\le\alpha y_1$ for $j\ge2$, so $A_k\subset\Omega_+$; moreover $|y|<2^{k+2}$. Let
$K_0=\lceil\log_2(2\sqrt{d-1})\rceil$. For $k\ge K_0$,
\[
|A_k|=2^k\bigl(2\lfloor\alpha2^k\rfloor+1\bigr)^{d-1}
\ge \alpha^{d-1}2^{kd},
\qquad |y|^{-d}\ge2^{-(k+2)d}\quad(y\in A_k).
\]
Consequently,
\[
\sum_{y\in A_k}|y|^{-d}\ge \alpha^{d-1}2^{-2d}=:w_d.
\]
Taking $K=\lfloor\log_2R\rfloor-2$ and summing over the disjoint sets $A_k$, $K_0\le k\le K$, gives
\[
\sum_{\substack{y\in\Omega\\0<|y|\le R}}|y|^{-d}
\ge (K-K_0+1)w_d\ge u_d\log R
\]
for all sufficiently large $R$. This proves the lemma after increasing the threshold $v_d$ if necessary.
\end{proof}

\subsection{Failure of strong type $(p,p)$}
Let
\[
B_N=[-N,N]^d\cap\mathbb Z^d,\qquad Q_N=[-N/2,N/2]^d\cap\mathbb Z^d,
\qquad f_N=(\#B_N)^{-1/2}\mathbf1_{B_N}.
\]
Then $\|f_N\|_2=1$. For $n\in Q_N$, positivity and the fact that $n-Q_N\subset B_N$ give
\[
\begin{aligned}
R_{\mathrm{dis}}^{(11)}f_N(n)
&\ge \frac1{(\#B_N)^{1/2}}\sum_{y\in Q_N\setminus\{0\}}K_{11}(y)
 \ge \frac{c_d}{d(\#B_N)^{1/2}}
 \sum_{\substack{y\in\Omega\\0<|y|\le N/2}}|y|^{-d} \\
&\ge \frac{c\log N}{(\#B_N)^{1/2}}
\end{aligned}
\]
for all sufficiently large $N$, where $c>0$ depends only on $d$. Therefore when $
N\rightarrow \infty ,
$
\[
\|R_{\mathrm{dis}}^{(11)}f_N\|_2^2
\ge \frac{c^2(\log N)^2\#Q_N}{\#B_N}
= c^2(\log N)^2\left(\frac{N+1}{2N+1}\right)^d
\ge \frac{c^2}{3^d}(\log N)^2\longrightarrow\infty.
\]
Thus $R_{\mathrm{dis}}^{(11)}$ is unbounded on $\ell^2$. Since its kernel is real and even, the operator is self-adjoint. If it were bounded on some $\ell^p$, then it would also be bounded on $\ell^{p'}$, and Riesz--Thorin interpolation would imply boundedness on $\ell^2$, a contradiction. Hence it is unbounded on every $\ell^p$, $1<p<\infty$.

\subsection{Failure of weak type $(1,1)$}
The endpoint failure uses the same positivity mechanism. Put
\[
g_N=(\#B_N)^{-1}\mathbf1_{B_N},\qquad \|g_N\|_1=1.
\]
For every $n\in Q_N$, the preceding calculation gives
\[
R_{\mathrm{dis}}^{(11)}g_N(n)
\ge \frac{c\log N}{\#B_N}.
\]
Choosing $\lambda_N=c\log N/(2\#B_N)$, we obtain when $
N\rightarrow \infty ,
$
\[
\begin{aligned}
\|R_{\mathrm{dis}}^{(11)}g_N\|_{\ell^{1,\infty}}
&\ge \lambda_N\#\{n:|R_{\mathrm{dis}}^{(11)}g_N(n)|>\lambda_N\} \\
&\ge \frac{c\log N}{2\#B_N}\#Q_N
\ge \frac{c\log N}{2\cdot3^d}\longrightarrow\infty.
\end{aligned}
\]
This completes the proof of Theorem~\ref{thm:diagonal}.

\section{The $\ell^2$ estimate for $jj-kk$}
We finally turn to the traceless combination $R_{\mathrm{dis}}^{(jj-kk)}$. Since the Fourier-analytic argument parallels Section~2, we present only the modifications needed here. By symmetry we take $j=1$ and $k=2$. The transference and multiplier arguments give
$$
\|R_{\mathrm{dis}}^{(11-22)}\|_{\ell^2\to\ell^2}=\lVert m_{11-22} \rVert _{L^{\infty}\left( \left[ 0,1 \right] ^d \right)}\ ,
$$
where $m_{11-22}(\xi) =\mathrm{p.v.} \sum_{m \in \mathbb{Z}^d \setminus \{0\}}c_d\frac{m_1^2-m_2^2}{|m|^{d+2}}e^{-2\pi i m \cdot \xi},~~~\xi \in [0, 1]^d.$ 

For convenience, set
\begin{equation*}
\begin{aligned}
    V_2(t, x) &= \sum_{n \neq 0} n^2 e^{-tn^2} \cos(2\pi n x), \\
    S_2(t, x) &= \sum_{n \in \mathbb{Z}} e^{-tn^2} \cos(2\pi n x),
\end{aligned}
\end{equation*}
and
$$\displaystyle M(t, \xi) = \sum_{m \in \mathbb{Z}^d \setminus \{0\}} (m_1^2-m_2^2) e^{-t|m|^2} \cos(2\pi m \cdot \xi).$$
We obtain
\begin{equation*}
 \begin{aligned}
    m_{11-22}(\xi) 
    &=  \frac{c_d}{\Gamma \left( \frac{d+2}{2} \right)}\int_0^{\infty}{t^{\frac{d}{2}}}M\left( t,\xi \right) dt \\
    &= \frac{c_d}{\Gamma \left( \frac{d+2}{2} \right)}\int_0^{\infty}{t^{\frac{d}{2}}}\left( V_2\left( t,\xi _1 \right) S_2\left( t,\xi _2 \right) -V_2\left( t,\xi _2 \right) S_2\left( t,\xi _1 \right) \right) \prod_{j=3}^d{S_2\left( t,\xi _j \right)}dt.
\end{aligned}
\end{equation*}
\subsection{Upper bound}
We first estimate $\|m_{11-22}\|_{L^\infty([0,1]^d)}$. Splitting at $t=1$,
\begin{equation*}
 \begin{aligned}
    |m_{11-22}(\xi) |
    &=  \frac{c_d}{\Gamma \left( \frac{d+2}{2} \right)}|\int_0^{\infty}{t^{\frac{d}{2}}}M\left( t,\xi \right) dt| \\
    &\le \frac{c_d}{\Gamma \left( \frac{d+2}{2} \right)}\left( \left| \int_0^1{t^{\frac{d}{2}}M\left( t,\xi \right) dt} \right|+\left| \int_1^{\infty}{t^{\frac{d}{2}}M\left( t,\xi \right)}dt \right| \right)  \\
    & := I_{small}+I_{large}.
\end{aligned}
\end{equation*}
We again treat the small- and large-time regimes separately.

For $t>1$, we have the following estimate
\begin{equation*}
\begin{aligned}
   \left| S_2\left( t,x \right) \right|&\le \sum_{m\in \mathbb{Z}}{e^{-tm^2}}\le 1+2e^{-t}+3e^{-4t},\\
   \left| V_2\left( t,x \right) \right|&\le \sum_{m\in \mathbb{Z}}{m^2e^{-tm^2}}\le 2e^{-t}+8e^{-4t}+28e^{-9t},
\end{aligned}
\end{equation*}
and
$$
I_{large}\le 2\frac{c_d}{\Gamma \left( \frac{d+2}{2} \right)}\int_1^{\infty}{t^{\frac{d}{2}}}\left( 2e^{-t}+8e^{-4t}+28e^{-9t} \right) \left( 1+2e^{-t}+3e^{-4t} \right) ^{d-1}dt.
$$
As in Section~2, the leading contribution is bounded by
$$
2\frac{c_d}{\Gamma \left( \frac{d+2}{2} \right)}\int_1^{\infty}{t^{\frac{d}{2}}}2e^{-t}\left( 1+2\left( d-1 \right) e^{-t} \right) dt\le 4c_d\left( 1+\frac{d-1}{2^{\frac{d}{2}}} \right) .
$$

For $0<t\leq1$, the Poisson-summation estimates from Section~2 show that $I_{\mathrm{small}}$ is lower order relative to the large-time contribution.

Combining the two regimes gives
$$
\|R_{\mathrm{dis}}^{(11-22)}\|_{\ell^2\to\ell^2}=\lVert m_{11-22} \rVert _{L^{\infty}}\le 4c_d\left( 1+\frac{\left( 1+o\left( 1 \right) \right) d}{2^{\frac{d}{2}}} \right) .
$$
\subsection{Lower bound}
For the lower bound, we evaluate the continuous multiplier at the half-period point $
\xi _0=\left( 0,1/2,0,\dots,0 \right) 
$, obtaining
\begin{equation*}
 \begin{aligned}
    |m_{11-22}(\xi_0) |
    &=  \frac{c_d}{\Gamma \left( \frac{d+2}{2} \right)}|\int_0^{\infty}{t^{\frac{d}{2}}}M\left( t,\xi_0 \right) dt| \\
    &\ge \frac{c_d}{\Gamma \left( \frac{d+2}{2} \right)}\left( \left| \int_1^{\infty}{t^{\frac{d}{2}}M\left( t,\xi_0 \right) dt} \right|-\left| \int_0^1{t^{\frac{d}{2}}M\left( t,\xi_0 \right)}dt \right| \right)  \\
    & := J_{large}-J_{small}.
\end{aligned}
\end{equation*}

For $t>1$, we have
\begin{equation*}
\begin{aligned}
   S_2\left( t,\frac{1}{2} \right) &=\sqrt{\frac{\pi}{t}}\sum_{k\in \mathbb{Z}}{e^{-\frac{\pi ^2\left( k+\frac{1}{2} \right) ^2}{t}}}>0 ,\\
    V_2\left( t,\frac{1}{2} \right) &=2\sum_{m=1}^{\infty}{\left( -1 \right) ^m}m^2e^{-tm^2}\le 2\left( -e^{-t}+4e^{-4t} \right) ,
\end{aligned}
\end{equation*}
and
\begin{equation*}
 \begin{aligned}
    M\left( t,\xi _0 \right)
    &\ge  \left( 2e^{-t}\left( 1-2e^{-t} \right) +\left( 2e^{-t}-8e^{-4t} \right) \left( 1+2e^{-t} \right) \right) \left( 1+2e^{-t} \right) ^{d-2} \\
    &= \left( 4e^{-t}-8e^{-4t}-16e^{-5t} \right) \left( 1+2e^{-t} \right) ^{d-2}.  \\
\end{aligned}
\end{equation*}
Consequently, $
J_{large}\ge \frac{c_d}{\Gamma \left( \frac{d}{2}+1 \right)}\int_1^{\infty}{t^{\frac{d}{2}}}\left( 4e^{-t}-8e^{-4t}-16e^{-5t} \right) \left( 1+2\left( d-2 \right) e^{-t} \right) dt.
$ The same large-time expansion as in Section~2 shows that the leading term on the right-hand side is $
4c_d+4\frac{\left( d-2 \right) c_d}{2^{\frac{d}{2}}}.
$

The small-time contribution $J_{\mathrm{small}}$ is again lower order by the estimates already established in Section~2.

Combining the two regimes gives
$$
\|R_{\mathrm{dis}}^{(11-22)}\|_{\ell^2\to\ell^2}=\lVert m_{11-22} \rVert _{L^{\infty}}\ge 4c_d\left( 1+\frac{\left( 1+o\left( 1 \right) \right) d}{2^{\frac{d}{2}}} \right) ,
$$
which completes the proof of Theorem~\ref{thm:diagonal-difference}.

\section{The $\ell^{1,\infty}$ and $\ell^p$ estimates for $jj-kk$}
In this section we prove Theorem~\ref{thm:diagonal-difference-lp}. Since most of the continuous--discrete comparison is parallel to Section~3, we emphasize only the points where the traceless kernel produces a different dimension-dependent main term. By symmetry, we take $j=1$ and $k=2$ and write
\[
T_{\mathrm{dis}}:=R_{\mathrm{dis}}^{(11-22)},
\qquad
K_{11-22}(m)=c_d\frac{m_1^2-m_2^2}{|m|^{d+2}}.
\]
Let
\[
K_{11-22}^*(x)=c_d\frac{x_1^2-x_2^2}{|x|^{d+2}}\chi_{\{|x|\geq1\}},
\]
and, for $z\in\mathbb Z^d$,
\[
\widetilde K_{11-22}^*(z)
:=\int_Q\int_Q\bigl(K_{11-22}^*(z+s-t)-K_{11-22}^*(z)\bigr)\,dt\,ds,
\qquad Q=\left[-\frac12,\frac12\right)^d.
\]
The main new ingredient is the following lattice-sum estimate.

\begin{lemma}\label{lem:jjkk-lattice-sum}
There exists an absolute constant $C>0$ such that
\[
\sum_{\substack{z\in\mathbb Z^d\\1\leq|z|\leq d^2}}
\frac{|z_1^2-z_2^2|}{|z|^{d+2}}
\leq 4+C\frac{d}{2^{d/2}}.
\]
Moreover, as $d\to\infty$,
\[
\sum_{\substack{z\in\mathbb Z^d\\1\leq|z|\leq d^2}}
\frac{|z_1^2-z_2^2|}{|z|^{d+2}}
=4+(4+o(1))\frac{d}{2^{d/2}}.
\]
\end{lemma}

\begin{proof}
Set
\[
A_d=\sum_{\substack{z\in\mathbb Z^d\\1\leq|z|\leq d^2}}
\frac{|z_1^2-z_2^2|}{|z|^{d+2}}.
\]
The first two lattice shells give the lower bound. The four points $\pm e_1,\pm e_2$ contribute $4$. On $|z|^2=2$, the only nonzero contributions are
\[
\pm e_1\pm e_r,\qquad \pm e_2\pm e_r,\qquad 3\leq r\leq d,
\]
so there are $8(d-2)$ such points. Hence
\begin{equation}\label{eq:jjkk-Ad-lower}
A_d\geq4+\frac{4(d-2)}{2^{d/2}}.
\end{equation}

For the upper bound,
\[
|z_1^2-z_2^2|\leq z_1^2+z_2^2,
\]
and coordinate symmetry gives
\begin{equation}\label{eq:jjkk-Ad-radial}
A_d
\leq\frac2d\sum_{\substack{z\in\mathbb Z^d\\1\leq|z|\leq d^2}}\frac1{|z|^d}
=: \frac2d S_d.
\end{equation}
We estimate $S_d$ by the same Gamma-integral method as before. Put $a=d/2$ and
\[
F_d(t)=\sum_{\substack{z\in\mathbb Z^d\\1\leq|z|\leq d^2}}e^{-t|z|^2}.
\]
Then
\begin{equation}\label{eq:jjkk-Sd-gamma}
S_d=\frac1{\Gamma(a)}\int_0^\infty t^{a-1}F_d(t)\,dt.
\end{equation}
Let $t_0=d^{-4}$ and $L=2\log d$, and split the integral into the four ranges
\[
(0,t_0),\qquad [t_0,1],\qquad [1,L],\qquad [L,\infty).
\]
For the small-time range, the truncation $|z|\leq d^2$ gives
\[
F_d(t)\leq(2d^2+1)^d\leq(3d^2)^d,
\]
and hence
\[
\frac1{\Gamma(a)}\int_0^{t_0}t^{a-1}F_d(t)\,dt
\leq\frac{3^d}{\Gamma(a+1)}
=o\left(\frac{d^2}{2^{d/2}}\right).
\]
On $[t_0,1]$, Poisson summation for $\Theta(t)=\sum_{n\in\mathbb Z}e^{-tn^2}$ gives the same small-time control as in Lemma~\ref{lem:lattice-sum}; on $[1,L]$, the original theta expansion is used instead. These two contributions satisfy, respectively,
\[
O\left(\frac{C^d\pi^{d/2}\log d}{\Gamma(d/2)}\right)
\quad\text{and}\quad
O\left(\frac{C^d(2\log d)^{d/2}}{\Gamma(d/2+1)}\right),
\]
and are both $o(d^2/2^{d/2})$ by Stirling's formula.

It remains to consider $t\geq L$, where the first two lattice shells become explicit. Since
\[
F_d(t)\leq \Theta(t)^d-1,
\qquad
\Theta(t)=1+2e^{-t}+O(e^{-4t}),
\]
and $de^{-t}=O(d^{-1})$ on this interval, the binomial expansion yields uniformly
\[
\Theta(t)^d-1
=2de^{-t}+2d(d-1)e^{-2t}
+O\left(d^3e^{-3t}+de^{-4t}+d^2e^{-5t}\right).
\]
Using
\[
\frac1{\Gamma(a)}\int_0^\infty t^{a-1}e^{-qt}\,dt=q^{-a}
\]
and enlarging the positive integrals from $[L,\infty)$ to $(0,\infty)$, we obtain
\[
S_d
\leq2d+\frac{2d(d-1)}{2^{d/2}}
+O\left(\frac{d^3}{3^{d/2}}+\frac{d}{4^{d/2}}+\frac{d^2}{5^{d/2}}\right)
+o\left(\frac{d^2}{2^{d/2}}\right).
\]
Thus
\begin{equation}\label{eq:jjkk-Sd-asymptotic}
S_d\leq2d+\frac{2d(d-1)}{2^{d/2}}+o\left(\frac{d^2}{2^{d/2}}\right).
\end{equation}
The reverse inequality with these two leading terms follows by retaining only the shells $|z|^2=1$ and $|z|^2=2$. Combining this with \eqref{eq:jjkk-Ad-radial} and \eqref{eq:jjkk-Ad-lower} gives
\[
A_d=4+(4+o(1))\frac{d}{2^{d/2}}.
\]
The non-asymptotic estimate follows after enlarging the absolute constant to cover the finitely many remaining dimensions.
\end{proof}

We now estimate the error kernel. Since the decomposition is the same as in Lemma~\ref{lem:error-kernel}, we keep the repeated estimates brief.

\begin{lemma}\label{lem:jjkk-error-kernel}
There exists an absolute constant $\eta>0$ such that
\[
\|\widetilde K_{11-22}^*\|_{\ell^1(\mathbb Z^d)}
\leq4c_d+\eta c_d\frac{d}{2^{d/2}}.
\]
Moreover, as $d\to\infty$,
\[
\|\widetilde K_{11-22}^*\|_{\ell^1(\mathbb Z^d)}
\leq4c_d\left[1+(1+o(1))\frac{d}{2^{d/2}}\right].
\]
\end{lemma}

\begin{proof}
For $|z|>d^2$,
\[
|\nabla K_{11-22}(x)|\leq c_d(d+4)|x|^{-d-1}.
\]
The mean value theorem and the same lattice-to-integral comparison used in Lemma~\ref{lem:error-kernel} therefore give
\begin{equation}\label{eq:jjkk-far-error}
\sum_{|z|>d^2}|\widetilde K_{11-22}^*(z)|
=O(\sqrt d)
=o\left(c_d\frac{d}{2^{d/2}}\right),
\end{equation}
where we used $c_d|\mathbb S^{d-1}|=d$ in the last step.

For $|z|\leq d^2$, set
\[
\Phi(u)=\prod_{r=1}^d(1-|u_r|)_+.
\]
As in Section~3,
\[
\sum_{|z|\leq d^2}|\widetilde K_{11-22}^*(z)|\leq I_d+II_d,
\]
where
\[
I_d=\sum_{|z|\leq d^2}\int_{\mathbb R^d}\Phi(u)|K_{11-22}^*(z+u)|\,du,
\qquad
II_d=\sum_{|z|\leq d^2}|K_{11-22}^*(z)|.
\]
Using $\sum_{z\in\mathbb Z^d}\Phi(x-z)=1$ exactly as in the off-diagonal proof,
\[
I_d
\leq c_d\log(d^2+\sqrt d)
\int_{\mathbb S^{d-1}}|\theta_1^2-\theta_2^2|\,d\sigma(\theta).
\]
Rotational symmetry in the $(\theta_1,\theta_2)$-plane gives
\[
\int_{\mathbb S^{d-1}}|\theta_1^2-\theta_2^2|\,d\sigma(\theta)
=\frac{4}{\pi d}|\mathbb S^{d-1}|,
\]
so
\begin{equation}\label{eq:jjkk-Id}
I_d\leq\frac4\pi\log(d^2+\sqrt d)
=o\left(c_d\frac{d}{2^{d/2}}\right).
\end{equation}
On the other hand, Lemma~\ref{lem:jjkk-lattice-sum} gives
\begin{equation}\label{eq:jjkk-IId}
II_d\leq4c_d\left[1+(1+o(1))\frac{d}{2^{d/2}}\right].
\end{equation}
Combining \eqref{eq:jjkk-far-error}, \eqref{eq:jjkk-Id}, and \eqref{eq:jjkk-IId} proves the asymptotic estimate. Enlarging the constant for finitely many small dimensions gives the first assertion.
\end{proof}

\subsection{The $\ell^p$ estimate}
Define
\[
R_1^{(11-22)}F(x)=\int_{\mathbb R^d}K_{11-22}^*(x-y)F(y)\,dy.
\]
Since $P(x)=x_1^2-x_2^2$ is a homogeneous harmonic polynomial of degree two, the higher-order truncation factorization used in Section~3 applies verbatim. By \cite{KKW},
\[
R_1^{(11-22)}F=b_{2,d}*R^{(11-22)}F,
\qquad
b_{2,d}=|B(0,1)|^{-1}\chi_{B(0,1)}.
\]
Dilation and the convergence of the truncated transforms then give
\begin{equation}\label{eq:jjkk-truncated-norm}
\|R_1^{(11-22)}\|_{L^p\to L^p}
=\|R^{(11)}-R^{(22)}\|_{L^p\to L^p}
=p^*-1,
\end{equation}
where the last equality is the sharp continuous estimate; see \cite{BK2}.

By the comparison argument of Proposition~6.1 in \cite{BKK},
\[
\|R_{\mathrm{dis}}^{(11-22)}\|_{\ell^p\to\ell^p}
\leq\|R_1^{(11-22)}\|_{L^p\to L^p}
+\|\widetilde K_{11-22}^*\|_{\ell^1}.
\]
Thus Lemma~\ref{lem:jjkk-error-kernel} yields
\begin{equation}\label{eq:jjkk-lp-upper}
\|R_{\mathrm{dis}}^{(11-22)}\|_{\ell^p\to\ell^p}
\leq(p^*-1)+4c_d\left[1+(1+o(1))\frac{d}{2^{d/2}}\right].
\end{equation}
The non-asymptotic upper estimate follows from the first part of the same lemma.

For the lower bound, $K_{11-22}$ is real-valued and even, hence $R_{\mathrm{dis}}^{(11-22)}$ is self-adjoint. Duality and Riesz--Thorin imply
\[
\|R_{\mathrm{dis}}^{(11-22)}\|_{\ell^2\to\ell^2}
\leq\|R_{\mathrm{dis}}^{(11-22)}\|_{\ell^p\to\ell^p}.
\]
Using Theorem~\ref{thm:diagonal-difference},
\begin{equation}\label{eq:jjkk-lp-lower}
\|R_{\mathrm{dis}}^{(11-22)}\|_{\ell^p\to\ell^p}
\geq4c_d\left(1+\frac{d-2}{2^{d/2}}-\frac{\rho}{4^{d/2}}\right).
\end{equation}
For fixed $p$, $p^*-1=o(c_dd/2^{d/2})$, and \eqref{eq:jjkk-lp-upper}--\eqref{eq:jjkk-lp-lower} give
\[
\|R_{\mathrm{dis}}^{(11-22)}\|_{\ell^p\to\ell^p}
=4c_d\left[1+(1+o(1))\frac{d}{2^{d/2}}\right].
\]

\subsection{The weak-type $(1,1)$ estimate}
For the lower bound, take $f=\delta_0$. At $\pm e_1$ and $\pm e_2$,
\[
|R_{\mathrm{dis}}^{(11-22)}f(n)|=c_d.
\]
Thus, with $\lambda=(1-\varepsilon)c_d$ and $\varepsilon \rightarrow 0$,
\begin{equation}\label{eq:jjkk-weak-lower}
\|R_{\mathrm{dis}}^{(11-22)}\|_{\ell^1\to\ell^{1,\infty}}\geq4c_d.
\end{equation}

For the upper bound, the proof of Lemma~\ref{lem:weak-transference} applies verbatim to $R_{\mathrm{dis}}^{(11-22)}$. For $f\in\ell^1(\mathbb Z^d)$, let
\[
F(x)=\sum_{n\in\mathbb Z^d}f(n)\mathbf1_Q(x-n).
\]
Writing
\[
\widetilde R_{\mathrm{dis}}^{(11-22)}F=R_1^{(11-22)}F+E,
\]
the calculation from Section~3 gives
\begin{equation}\label{eq:jjkk-error-E}
\|E\|_{L^1}\leq\|\widetilde K_{11-22}^*\|_{\ell^1}\|f\|_{\ell^1}.
\end{equation}
The standard weak-type estimate for maximal truncations of Calder\'on--Zygmund operators, applied to $K_{11-22}$, gives the crude dimensional bound
\begin{equation}\label{eq:jjkk-continuous-weak}
\|R_1^{(11-22)}F\|_{L^{1,\infty}}\leq C_0^d d^2\|F\|_{L^1}
\end{equation}
for an absolute constant $C_0>1$; see the classical Calder\'on--Zygmund argument in \cite{CZ}. Indeed, the usual covering proof has at most exponential dependence on $d$, while the normalized kernel satisfies $c_d|B(0,1)|=1$ and $|\nabla K_{11-22}(x)|\lesssim d c_d|x|^{-d-1}$. Thus the rough bound above is more than sufficient for the asymptotic below.

Put $\varepsilon_d=3^{-d/2}$. Since
\[
\{|E+R_1^{(11-22)}F|>\lambda\}
\subset
\left\{|E|>\frac{\lambda}{1+\varepsilon_d}\right\}
\cup
\left\{|R_1^{(11-22)}F|>\frac{\varepsilon_d\lambda}{1+\varepsilon_d}\right\},
\]
Chebyshev's inequality and \eqref{eq:jjkk-error-E}--\eqref{eq:jjkk-continuous-weak} imply
\[
\begin{aligned}
\|R_{\mathrm{dis}}^{(11-22)}f\|_{\ell^{1,\infty}}
&\leq(1+\varepsilon_d)\|\widetilde K_{11-22}^*\|_{\ell^1}\|f\|_{\ell^1}\\
&\quad+\frac{1+\varepsilon_d}{\varepsilon_d}C_0^d d^2\|f\|_{\ell^1}.
\end{aligned}
\]
By Stirling's formula,
\[
(1+3^{d/2})C_0^d d^2=o\left(c_d\frac{d}{2^{d/2}}\right),
\]
and Lemma~\ref{lem:jjkk-error-kernel} therefore gives
\begin{equation}\label{eq:jjkk-weak-upper}
\|R_{\mathrm{dis}}^{(11-22)}\|_{\ell^1\to\ell^{1,\infty}}
\leq4c_d\left[1+(1+o(1))\frac{d}{2^{d/2}}\right].
\end{equation}
The non-asymptotic upper bound follows after enlarging an absolute constant to cover the finitely many small dimensions. Together with \eqref{eq:jjkk-weak-lower}, this completes the proof of Theorem~\ref{thm:diagonal-difference-lp}.

\section{Appendix}

In this section, we collect several Fourier-analytic identities used in
Sections~2, 5, and~6. Throughout the appendix, we identify
$\mathbb T^d$ with $[0,1]^d$ and use the Fourier transform convention
\[
\widehat{F}(\xi)
=
\int_{\mathbb R^d}F(x)e^{-2\pi i x\cdot\xi}\,dx.
\]

For $n\in\mathbb Z^d\setminus\{0\}$, we write
\[
K_{12}(n)
=
c_d\frac{n_1n_2}{|n|^{d+2}}
\]
and
\[
K_{11-22}(n)
=
c_d\frac{n_1^2-n_2^2}{|n|^{d+2}}.
\]

\begin{lemma}\label{lem:multiplier}
The Fourier multiplier of the continuous-discrete operator
$\widetilde{R}_{\mathrm{dis}}^{(12)}$ is
\begin{equation}\label{eq:appendix-m12}
m_{12}(\xi)
=
\mathrm{p.v.}
\sum_{n\in\mathbb Z^d\setminus\{0\}}
K_{12}(n)e^{-2\pi i n\cdot\xi}
=
\mathrm{p.v.}
\sum_{n\in\mathbb Z^d\setminus\{0\}}
c_d\frac{n_1n_2}{|n|^{d+2}}
e^{-2\pi i n\cdot\xi}.
\end{equation}
Similarly, the Fourier multiplier associated with
$R_{\mathrm{dis}}^{(11-22)}
=R_{\mathrm{dis}}^{(11)}-R_{\mathrm{dis}}^{(22)}$
is
\begin{equation}\label{eq:appendix-m1122}
m_{11-22}(\xi)
=
\mathrm{p.v.}
\sum_{n\in\mathbb Z^d\setminus\{0\}}
K_{11-22}(n)e^{-2\pi i n\cdot\xi}
\end{equation}
that is,
\[
m_{11-22}(\xi)
=
\mathrm{p.v.}
\sum_{n\in\mathbb Z^d\setminus\{0\}}
c_d\frac{n_1^2-n_2^2}{|n|^{d+2}}
e^{-2\pi i n\cdot\xi}.
\]
Here the principal values are understood as the $L^2(\mathbb T^d)$
limits of the corresponding radially truncated Fourier series.
\end{lemma}

\begin{proof}
We give the details for $K_{12}$; the difference kernel is identical. Since
\[
|K_{12}(n)|\le c_d|n|^{-d},
\qquad
\sum_{n\ne0}|K_{12}(n)|^2<\infty,
\]
the radially truncated Fourier series
\[
m_{12,N}(\xi)=\sum_{0<|n|\le N}K_{12}(n)e^{-2\pi i n\cdot\xi}
\]
forms a Cauchy sequence in $L^2(\mathbb T^d)$, because Parseval gives
\[
\|m_{12,N}-m_{12,M}\|_2^2
=\sum_{M<|n|\le N}|K_{12}(n)|^2\longrightarrow0.
\]
Its limit is therefore the principal-value series in \eqref{eq:appendix-m12}. If
$\mu_{12}=\sum_{n\ne0}K_{12}(n)\delta_n$, then
$\widetilde R_{\mathrm{dis}}^{(12)}F=\mu_{12}*F$ for $F\in\mathcal S(\mathbb R^d)$, and hence
\[
\widehat{\widetilde R_{\mathrm{dis}}^{(12)}F}(\xi)
=m_{12}(\xi)\widehat F(\xi).
\]
For $K_{11-22}$ we use
$|K_{11-22}(n)|\le c_d|n|^{-d}$ and repeat the same argument, obtaining
\eqref{eq:appendix-m1122}.
\end{proof}

\begin{lemma}\label{lem:gaussian-representation}
Let
\[
q_{12}(m)=m_1m_2,
\qquad
q_{11-22}(m)=m_1^2-m_2^2,
\]
and let $q$ denote either one of these two quadratic forms. Then
\begin{align}
&\sum_{m\in\mathbb Z^d\setminus\{0\}}
q(m)\cos(2\pi m\cdot\xi)
\int_0^\infty
t^{d/2}e^{-t|m|^2}\,dt
\nonumber\\
&\qquad =
\lim_{\varepsilon\to0}
\sum_{m\in\mathbb Z^d\setminus\{0\}}
q(m)\cos(2\pi m\cdot\xi)
\int_\varepsilon^\infty
t^{d/2}e^{-t|m|^2}\,dt
\label{eq:appendix-L2-limit}
\end{align}
in $L^2(\mathbb T^d)$. Moreover,
\begin{align}
&\sum_{m\in\mathbb Z^d\setminus\{0\}}
q(m)\cos(2\pi m\cdot\xi)
\int_0^\infty
t^{d/2}e^{-t|m|^2}\,dt
\nonumber\\
&\qquad =
\lim_{\varepsilon\to0}
\int_\varepsilon^\infty
t^{d/2}
\sum_{m\in\mathbb Z^d\setminus\{0\}}
q(m)e^{-t|m|^2}
\cos(2\pi m\cdot\xi)\,dt
\label{eq:appendix-integral-interchange}
\end{align}
in $L^2(\mathbb T^d)$.

In particular, if
\[
S_1(t,x)
=
\sum_{n\in\mathbb Z}
n e^{-tn^2}\sin(2\pi nx)
\]
and
\[
S_2(t,x)
=
\sum_{n\in\mathbb Z}
e^{-tn^2}\cos(2\pi nx),
\]
then
\begin{equation}\label{eq:appendix-m12-integral}
m_{12}(\xi)
=
-\frac{c_d}{\Gamma\left(\frac{d+2}{2}\right)}
\int_0^\infty
t^{d/2}
S_1(t,\xi_1)S_1(t,\xi_2)
\prod_{j=3}^d S_2(t,\xi_j)\,dt.
\end{equation}

Furthermore, setting
\[
V_2(t,x)
=
\sum_{n\in\mathbb Z}
n^2e^{-tn^2}\cos(2\pi nx),
\]
we have
\begin{align}
m_{11-22}(\xi)
&=
\frac{c_d}{\Gamma\left(\frac{d+2}{2}\right)}
\int_0^\infty
t^{d/2}
\Big(
V_2(t,\xi_1)S_2(t,\xi_2)
-
V_2(t,\xi_2)S_2(t,\xi_1)
\Big)
\prod_{j=3}^dS_2(t,\xi_j)\,dt.
\label{eq:appendix-m1122-integral}
\end{align}
All the identities above are understood in the $L^2(\mathbb T^d)$
sense.
\end{lemma}

\begin{proof}
For $\varepsilon>0$ put
\[
b_m^\varepsilon=q(m)\int_\varepsilon^\infty t^{d/2}e^{-t|m|^2}\,dt,
\qquad
b_m=q(m)\int_0^\infty t^{d/2}e^{-t|m|^2}\,dt.
\]
The change of variables $u=t|m|^2$ gives
\[
b_m=\Gamma\left(\frac{d+2}{2}\right)\frac{q(m)}{|m|^{d+2}},
\qquad |b_m|\le \Gamma\left(\frac{d+2}{2}\right)|m|^{-d}.
\]
Thus $(b_m)_{m\ne0}\in\ell^2$, and dominated convergence in $\ell^2$, followed by Parseval, yields
\[
\sum_{m\ne0}b_m^\varepsilon\cos(2\pi m\cdot\xi)
\longrightarrow
\sum_{m\ne0}b_m\cos(2\pi m\cdot\xi)
\quad\text{in }L^2(\mathbb T^d).
\]
For fixed $\varepsilon>0$ the Gaussian decay makes the sum absolutely integrable, so Fubini's theorem permits the interchange of the sum and the integral over $(\varepsilon,\infty)$. Letting $\varepsilon\downarrow0$ proves \eqref{eq:appendix-L2-limit} and \eqref{eq:appendix-integral-interchange}.

It remains only to factor the Gaussian sums. For $q(m)=m_1m_2$, parity gives
\[
\sum_{m\ne0}m_1m_2e^{-t|m|^2}\cos(2\pi m\cdot\xi)
=-S_1(t,\xi_1)S_1(t,\xi_2)\prod_{j=3}^dS_2(t,\xi_j),
\]
which yields \eqref{eq:appendix-m12-integral}. Likewise,
\[
\sum_{m\ne0}(m_1^2-m_2^2)e^{-t|m|^2}\cos(2\pi m\cdot\xi)
=\bigl(V_2(t,\xi_1)S_2(t,\xi_2)-V_2(t,\xi_2)S_2(t,\xi_1)\bigr)
\prod_{j=3}^dS_2(t,\xi_j),
\]
and \eqref{eq:appendix-m1122-integral} follows.
\end{proof}

\begin{lemma}\label{lem:poisson-sums}
Let $t>0$ and $x\in[0,1]$. Define
\begin{align}
S_1(t,x)
&=
\sum_{n\in\mathbb Z}
n e^{-tn^2}\sin(2\pi nx),
\label{eq:S1_def}
\\
S_2(t,x)
&=
\sum_{n\in\mathbb Z}
e^{-tn^2}\cos(2\pi nx),
\label{eq:S2_def}
\\
V_2(t,x)
&=
\sum_{n\in\mathbb Z}
n^2e^{-tn^2}\cos(2\pi nx).
\label{eq:V2_def}
\end{align}
Then the Poisson summation formula gives
\begin{align}
S_1(t,x)
&=
\frac{\pi^{3/2}}{t^{3/2}}
\sum_{k\in\mathbb Z}
(x-k)
e^{-\frac{\pi^2(x-k)^2}{t}},
\label{eq:S1_dual}
\\
S_2(t,x)
&=
\sqrt{\frac{\pi}{t}}
\sum_{k\in\mathbb Z}
e^{-\frac{\pi^2(x-k)^2}{t}},
\label{eq:S2_dual}
\\
V_2(t,x)
&=
\frac{\sqrt{\pi}}{2t^{3/2}}
\sum_{k\in\mathbb Z}
\left(
1-\frac{2\pi^2(x-k)^2}{t}
\right)
e^{-\frac{\pi^2(x-k)^2}{t}}.
\label{eq:V2_dual}
\end{align}
\end{lemma}

\begin{proof}
We first prove \eqref{eq:S2_dual}. Since the summand is even,
\[
S_2(t,x)
=
\sum_{n\in\mathbb Z}
e^{-tn^2}e^{2\pi inx}.
\]
Let
\[
f(y)=e^{-ty^2}.
\]
With our Fourier transform convention,
\[
\widehat f(\xi)
=
\int_{\mathbb R}
e^{-ty^2}e^{-2\pi i\xi y}\,dy
=
\sqrt{\frac{\pi}{t}}
e^{-\frac{\pi^2\xi^2}{t}}.
\]
The Poisson summation formula
\[
\sum_{n\in\mathbb Z}
f(n)e^{2\pi inx}
=
\sum_{k\in\mathbb Z}
\widehat f(k-x)
\]
therefore yields
\[
S_2(t,x)
=
\sqrt{\frac{\pi}{t}}
\sum_{k\in\mathbb Z}
e^{-\frac{\pi^2(k-x)^2}{t}},
\]
which is \eqref{eq:S2_dual}.

Since the series and all of its derivatives converge absolutely and
locally uniformly for $t>0$, we may differentiate term by term.
From \eqref{eq:S2_def},
\[
\frac{\partial}{\partial x}S_2(t,x)
=
-2\pi
\sum_{n\in\mathbb Z}
n e^{-tn^2}\sin(2\pi nx)
=
-2\pi S_1(t,x).
\]
On the other hand, differentiating \eqref{eq:S2_dual}, we obtain
\[
\frac{\partial}{\partial x}S_2(t,x)
=
-\frac{2\pi^{5/2}}{t^{3/2}}
\sum_{k\in\mathbb Z}
(x-k)e^{-\frac{\pi^2(x-k)^2}{t}}.
\]
Consequently,
\[
S_1(t,x)
=
\frac{\pi^{3/2}}{t^{3/2}}
\sum_{k\in\mathbb Z}
(x-k)e^{-\frac{\pi^2(x-k)^2}{t}},
\]
which proves \eqref{eq:S1_dual}.

Finally, differentiating \eqref{eq:S2_def} with respect to $t$ gives
\[
-\frac{\partial}{\partial t}S_2(t,x)
=
\sum_{n\in\mathbb Z}
n^2e^{-tn^2}\cos(2\pi nx)
=
V_2(t,x).
\]
Differentiating the right-hand side of
\eqref{eq:S2_dual} with respect to $t$, we find
\begin{align*}
-\frac{\partial}{\partial t}S_2(t,x)
&=
\frac{\sqrt{\pi}}{2t^{3/2}}
\sum_{k\in\mathbb Z}
\left(
1-\frac{2\pi^2(x-k)^2}{t}
\right)
e^{-\frac{\pi^2(x-k)^2}{t}}.
\end{align*}
Thus \eqref{eq:V2_dual} follows and the proof is complete.
\end{proof}

\bibliographystyle{amsalpha}

\end{document}